\documentclass[10pt]{iopart}

\usepackage{iopams}
\usepackage{setstack}
\usepackage{color}
\usepackage{graphicx}
\usepackage[caption=false]{subfig}
\usepackage{hyphenat}

\usepackage{geometry}
\usepackage{float}
\usepackage{placeins}

\usepackage{amsthm}
\usepackage[verbose]{wrapfig}
\usepackage{fancyhdr}
\usepackage{bm}
\usepackage{microtype}
\usepackage{hyperref}

\usepackage[textsize=tiny, linecolor=red,bordercolor=black,
            backgroundcolor=red,textwidth=2.5cm]{todonotes}

\newcounter{todocounter}

\usepackage{MathMacros} 

\newtheorem{theorem}{Theorem}
\newtheorem{proposition}[theorem]{Proposition}
\newtheorem{lemma}[theorem]{Lemma}
\newtheorem{definition}[theorem]{Definition}
\newtheorem{assumption}[theorem]{Assumption}
\newtheorem{remark}[theorem]{Remark}
\newtheorem*{notation*}{Notation}

\begin{document}

\title[Tikhonov functionals with a tolerance measure introduced in the regularization]{Tikhonov functionals with a tolerance measure introduced in the regularization}
\author{Iwona Piotrowska-Kurczewski$^1$   and \underline{Georgia Sfakianaki}$^{1}$}
\address{$^1$ Center for Industrial Mathematics, University of Bremen, Bremen, Germany}
\ead{\mailto{ipiotrow@uni-bremen.de}, \mailto{gsfakian@uni-bremen.de}}

\begin{abstract}
   We consider a modified Tikhonov-type functional for the solution of ill-posed nonlinear inverse problems. Motivated by applications in the field of production engineering, we allow small deviations in the solution, which are modeled through a tolerance measure in the regularization term of the functional. The existence, stability and weak convergence of minimizers are proved for such a functional, as well as the convergence rates in the Bregman distance. We present an example for illustrating the effect of tolerances on the regularized solution and examine parameter choice rules for finding the optimal regularization parameter for the assumed tolerance value. In addition, we discuss the prospect of reconstructing sparse solutions when tolerances are incorporated in the regularization functional.
\end{abstract} 

\vspace{2pc}
	 
%
%
%
\section{Introduction}

The  classical inverse problem  is described by an operator equation of the form
\begin{equation}\label{invpb}
    F(\ud) = v, 
\end{equation}
where $F$ is a linear or non-linear operator between some Hilbert/Banach spaces $U$ and $V$. In the case of ill-posedness, we resort to regularization methods for approximating the true solution $\ud$. The most developed and widely used method for solving ill-posed inverse problems is Tikhonov regularization, see~\cite{Tikhonov1977solutions,Tikhonov1998nonlinear}. Some of the classical results on Tikhonov regularization can be found in~\cite{Bonesky_etal_2008,EnglHankeNeubauer1996,Grasmair2008,Hofmann_et_al_2007,JinMaass,kaltenbacher2008iterative,Scherzer2008,Schuster2012}. Here, the regularized solution $\uad$ is defined as the minimizer of the Tikhonov functional
\begin{equation}\label{Tikhonov}
    \T_\alpha^\delta(u)=\|F(u)-v^{\delta}\|_V^p+\alpha \Rq(u), 
\end{equation} 
which consists of a discrepancy and a regularization term (also called penalty term). Through the regularization term we are able to include a-priori knowledge about the true solution. 

In recent years, the concept of sparsity is considered a powerful tool, especially in applications, see for instance~\cite{Bredies2009,DDD,Grasmair2008,JinMaass,RamlauTeschke2006}. In this case the true solution has a sparse representation in the given  basis or frame for the parameter space $U$, i.e., only a few coefficients are different from zero. It turns out that in many applications one has to choose between classical and sparse regularization. The new challenge, resulting from real-world applications, is to allow some deviations in the data $v^{\delta}$. In~\cite{Gralla2020} Tikhonov functionals incorporating tolerances in the discrepancy term were studied for the solution of inverse problems. The authors proposed an altered Tikhonov functional of the form
\begin{equation}\label{Tikhonove}
    \T_{\alpha,\ve}^\delta(u)=\norm{ d_{\ve}\left(F(u)-v^{\delta}\right)}_V^p+\alpha \Rq (u),
\end{equation} 
where $d_\ve(\cdot)$ denotes the $\varepsilon$-insensitive distance $d_{\ve}(\cdot)=\max\{\abs{\;\cdot\;} - \ve,0\}$. This approach makes sense, e.g., in production engineering. In the case of surface treatment, tolerances for the quality of the end product or for the measurement accuracy are often specified. These  methods have been successfully applied to the problem of process design in micro production and applications in image processing. In addition to the original reference, we refer the user to~\cite{Gralla2018_ICNFT} and~\cite{Gralla2018}, too. For linear operators the case $\varepsilon >0$ and $p=q=1$ is a generalization of Support Vector Regression (SVR) which can be used for treating ill-posed inverse problems, see for instance~\cite{smola2004tutorial}. Furthermore, in~\cite{krebs2011support} a rigorous analysis incorporating discrepancy terms with tolerance for solving linear integral equations was presented, under a semi-discrete setting in reproducing kernel Hilbert spaces (RKHS). 


Inspired by the great potential of such approaches in applications, in our work we examine the effect of tolerances in the regularization term of Tikhonov functionals. Including these inside the penalty term means that the solution will eventually lie inside a confidence interval. An application of interest is the development of new structural materials. In this case, the goal is to find appropriate values for a set of production parameters, like chemical composition, heating or cooling, to finally obtain materials satisfying certain properties. The desired properties of the new materials are given in the form of intervals, or in the form of a so-called performance profile, for further reading refer to~\cite{Otero2018}.

\subsection{Regularization functional with tolerances}

As discussed in the introduction, the $\ve$-\emph{insensitive function} $d_{\varepsilon}$ comes from the theory of SVR, for further reading see~\cite{krebs2011support, Scho2002, Vapnik99}, and was first introduced by Cortes and Vapnik in~\cite{CortesVapnik1995}. For a given $\varepsilon \geq 0$ the function $d_\varepsilon:\, \R \rightarrow \R$ is defined as 
\begin{eqnarray}\label{equ:deps}
    d_\varepsilon(x) := {\abs{x}}_\varepsilon = \max\lbrace \vert x \vert-\varepsilon,\, 0\rbrace.
\end{eqnarray}
In~\Fref{F:figure1a}, $\de$ as given in~\eref{equ:deps} is plotted in comparison to the absolute value function while~\Fref{F:figure1b} shows their subdifferentials. In the following, we often use the term \emph{tolerance function} when referring to the $\ve$-insensitive function. Two analogous definitions are used within this work which differ in $\varepsilon$ being a sequence or a function. We follow the definition in~\cite[Definition 1]{Gralla2020} and define the $\varepsilon$-insensitive modulus $d_{\varepsilon,n}$. 
\begin{definition}[$\varepsilon$-modulus function] \label{def:eps_modulus} 
    For $0 <\varepsilon \in \R^n$ we define the $\varepsilon$-insensitive modulus  $d_{\varepsilon,n}:\, \R^n \rightarrow \R^n$  component wise as
    \begin{eqnarray}\label{equ:deps_1}
	    d_{\varepsilon,n}(x)_i := d_{\varepsilon_i}(x_i), \quad i=1,\ldots,n.
    \end{eqnarray}
    For $\varepsilon: \Omega \rightarrow \R^n$, with $0 < \varepsilon \in L_q(\Omega)^n$ we define the $\varepsilon$-insensitive modulus function  $d_{\varepsilon,\Omega}:\, L_q(\R^n) \rightarrow L_q(\R^n)$ by  
    \begin{eqnarray}\label{equ:deps_2}
	    d_{\varepsilon,\Omega}(f)(\cdot):=d_{\varepsilon,n}(f(\cdot)).
    \end{eqnarray}
\end{definition}

\begin{figure}
	\begin{center}%
		\subfloat[]
		{\label{F:figure1a}
		\includegraphics[width=0.45\textwidth]{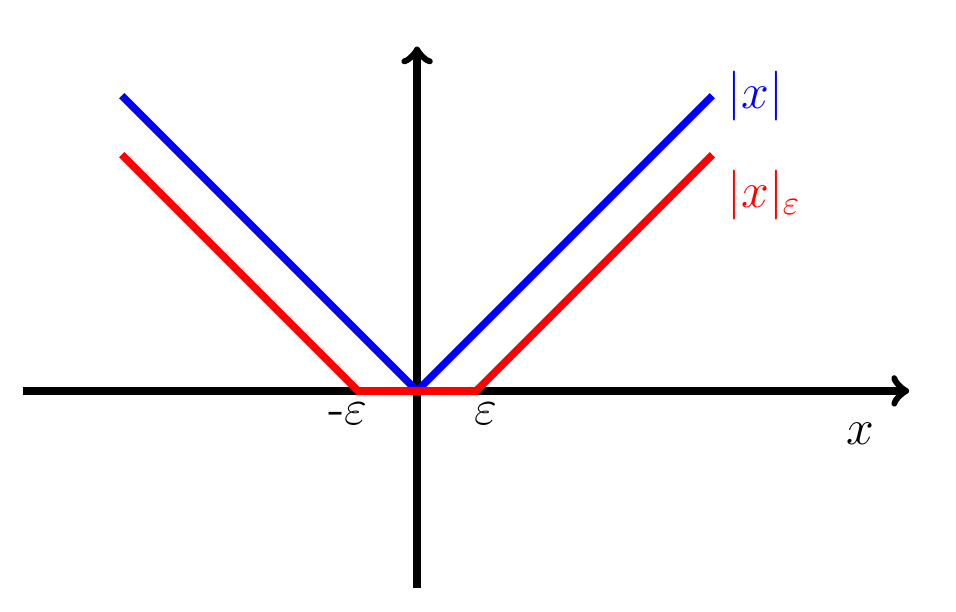}}
		\subfloat[]
		{\label{F:figure1b}
		\includegraphics[width=0.45\textwidth]{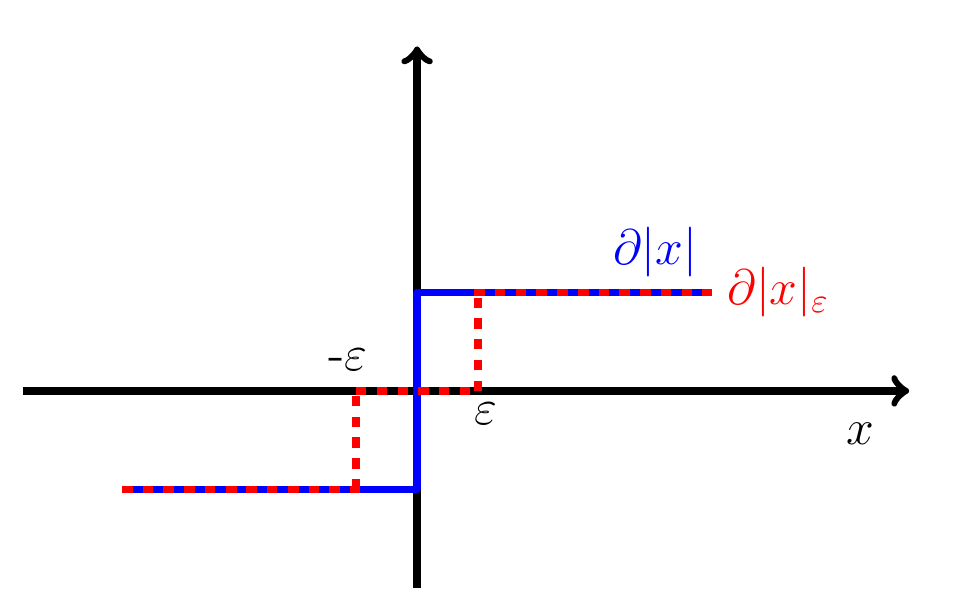}}
		\caption{The tolerance function $d_\varepsilon(x)=\abs{x}_\ve$ in comparison with the absolute value $\abs{x}$ for $x \in \R$ and $\ve >0$ in~(a), and their subdifferentials in (b).}    
		\label{F:figure1}
	\end{center} 
\end{figure}

For simplicity of notation, we write $d_{\varepsilon}$ for all cases.  In both definitions given in \eref{equ:deps_1} and \eref{equ:deps_2} the equation \eref{equ:deps} is applied point-wise. Analogously using the Definition~\ref{def:eps_modulus} point-wise in the $L_q$-induced norm we obtain a distance function in $L_q(\Omega)$ space.
\begin{definition}[$L_{q,\varepsilon}$-insensitive measure] \label{Def1Lqe_distance}
Let $\Omega\subset\R^n$ be a bounded and closed  and let $\varepsilon\in L_q(\Omega)$. The $L_{q,\varepsilon}$-insensitive measure is denoted via
    \begin{eqnarray}\label{equ:Lqeps}
	    \left\| u | L_q(\Omega)\right\|_\varepsilon = \|u\|_{L_{q,\varepsilon}} =  \| u \|_{q,\, \varepsilon} := \left(\int_\Omega d_\varepsilon (u(x))^q\; \dx\right)^\frac{1}{q}.
    \end{eqnarray}
\end{definition}

Our definition agrees with the one given in~\cite{Gralla2020}, and for the case of $L_q(\R^n)$ we further have to assume that $\varepsilon$ is bounded. For notational simplicity of our subsequent analysis, $\norm{\cdot}_{L_{q,\varepsilon}}$ will often be denoted by $\norm{\cdot}_{q,\varepsilon}$.

In regularization methods we often assume a reference solution which is included in the penalty term as a-priori information on the true solution of the problem. Denoting with $u^* \in L_q(\Omega)$ the reference solution and assuming including the tolerances, our penalty term is of the form 
\begin{equation}\label{penatlyRq-eps}
    \Rqe(u) := \norm{u - u^*}^q_{q,\ve} = \int_\Omega \left(\max \left\{ \abs{u(x)-u^*(x)} - \ve, 0 \right\}\right) ^q\;\mathrm{d}x,   
\end{equation}
where $1\leq q \leq 2$ and $\Omega$ bounded set in $\R^n$. Since $u^*$ does not affect our theoretical analysis, for simplicity, we assume it to be zero and we only consider it later in our numerical results. 

The functional $\Rqe$ is weakly lower semi-continuous and fulfills the following inequalities 
\begin{eqnarray}
    \norm{u}_{L_{q,\ve}} &\leq \norm{u}_{L_q}, \label{ineq1} \\
    \norm{u}_{L_q} &\leq \norm{u}_{L_{q, \ve}} + \norm{\ve}_{L_q}, \label{ineq2}
\end{eqnarray} 
which have been proved in~\cite{Gralla2020}. Furthermore, $\Rqe$ is continuous, convex for $q \geq 1$, whereas for $q>1$ is strictly convex. By \eref{ineq1} it is obvious that $d_\ve(u) \in L_q(\Omega)$ and, therefore, $\Rqe$ is well defined.

\begin{proposition}
	Let $\ve \in L_q(\Omega)$ for $1\leq q\leq 2$. The regularization functional $\Rqe(u) $  given by \eref{penatlyRq-eps} is coercive.
\end{proposition}
\begin{proof}
    This follows directly from the inequality~\eref{ineq1} since taking \(\norm{u}_{L_q} \to \infty\) leads to the conclusion that 
    \(\Rqe(u) = \norm{u}_{L_{q,\ve}}^q \to \infty \).%
\end{proof}

\subsection{Tikhonov functional with tolerance in regularization term}

Assuming $U = L_q(\Omega)$ over a bounded set $\Omega \subset \Rn$ and $V$ to be a reflexive Banach space, we consider an altered Tikhonov functional including the \emph{tolerance function} described in the previous section in the regularization term, that is
\begin{equation}\label{TikhonovRe}
    \Jevd(u):= \norm{F(u) - v^\delta}^p_V + \alpha \Rqe(u).
\end{equation}
Here $F : \dom (F)\subset U \to V$ is a nonlinear operator between  $U $ and $ V$ and the noisy data $\vd = v + n(\delta)$ are created with additive noise with level noise $\delta>0$ and are such that $\norm{v-\vd}_V \leq \delta$. The regularization term  $\Rqe : U \to \R_+$ for $1\leq q\leq 2$ includes the tolerance $\ve$ and  is  given by \eref{penatlyRq-eps}. We aim at investigating the analytical properties of minimizers $\uade$. Moreover, we examine the connection between tolerances in parameter space and sparsity regularization.
The following assumption remains valid throughout the paper. 

\begin{assumption}\label{AssumptionOnF}
    \begin{itemize}
        \item[(i)] Let $F : \Df \subset U \to V$ be weakly sequentially closed with respect to the weak topology on $U$. 
        \item[(ii)] The set $\D := \dom(F) \cap \dom(\Rqe)$ is non-empty. Note that this assumption implies that $\Rqe$ is proper.
    \end{itemize}
\end{assumption}

Furthermore, in the proofs of convergence and convergence rates of the minimizers of $\Jevd$, we use the concept of an $\mathcal{R}_{(\cdot)}$-minimizing solution. 

\begin{definition}[$\mathcal{R}$-minimizing solution] \label{Rminsol}
	The element $\ud \in U$ is called an $\mathcal{R}$-minimizing solution, if $F(\ud) = v$ and \(\mathcal{R}(\ud) = \min\limits_{u \in U} \left\{\mathcal{R}(u) : F(u) = v\right\}\).
\end{definition}

\section{Well-posedness}
We begin with the existence of minimizers $\uade:= \arg\min \Jevd$. Then, we continue with results on the stability of minimizers i.e., we prove that the minimizer depends continuously on the data. In the following results we use the next lemma which can be found in~\cite{Grasmair2008}.

\begin{lemma}\label{lemma6}
    Let $\ukseq \subset \dom(F)$. Assume that $\ve >0$ is fixed, $\vkseq \subset V$ is a bounded sequence in $V$ and that there exist $\alpha > 0$ and $M > 0$ such that $\Jevk(u_k) < M$, for all $k \in \N$. Then, there exist $\ut \in \dom(F)$ and a subsequence $\ukjseq$ such that $\ukj \wconv \ut$ and $F(\ukj) \wconv F(\ut)$. 
\end{lemma}
\begin{proof}
    The proof of this Lemma is omitted as it follows with similar steps as in~\cite[Lemma 4]{Grasmair2008}.
\end{proof}

In the theorems, we closely follow the concept in~\cite{Grasmair2008} and~\cite{kaltenbacher2008iterative} and prove them for our Tikhonov functional with tolerances incorporated in the regularization term.
\begin{theorem}[Existence]
	Assume that $\ve > 0$ is fixed. For $\alpha > 0$ and for every $\vd \in V$ the functional $\Jevd$ has a minimizer $\uade$ in $\D$.
\end{theorem}

\begin{proof}
	Let $\ukseq$ satisfy \(\limk \Jevd(\uk) = \inf \left\{  \Jevd(u) : u \in D \right\}\).
	From Lemma \ref{lemma6}, there exists a subsequence $\ukjseq$ weakly converging to some $\ut \in \dom(F)$ such that $F(\ukj) \wconv F(\ut)$. From the weak lower semi-continuity of $\Rqe$ and $\norm{\cdot}^p_V$ and the fact that $F$ is weakly sequentially closed it follows that
	\begin{eqnarray*}
		 \Jevd(\ut) &\leq \liminfj \norm{F(\ukj) - \vd}^p + \alpha \liminfj \Rqe(\ukj) \\
			    &\leq \liminfj \left\{\norm{F(\ukj) - \vd}^p + \alpha \Rqe(\ukj) \right\} \\
			    &\leq \limsupj \Jevd(\ukj), \quad \forall u \in \dom(F).
	\end{eqnarray*}
	Therefore, $ \Jevd(\ut) \leq  \Jevd(u)$ for any $u \in \dom(F)$, which means that $\uade:=\ut \in \dom(F)$ is a minimizer of $ \Jevd$.
\end{proof}

\begin{notation*}
    If any of the ingredients $\vd, \alpha, \ve$ is taken as a (sub)sequence, the functional will be denoted including the respective (sub)sequence in its shorthand notation, e.g., given a sequence of noisy data $\vk$, we will write $\Jevk$ for denoting the functional $\Jevk(u):= \norm{F(u) - \vk}^p_V + \alpha \Rqe(u)$.
\end{notation*}
    
The next theorem concerns the stability of minimizers of $\Jevd$, namely, for fixed $\alpha >0$ we prove that the minimizer $\uade$ depends continuously on $\vd$. 

\begin{theorem}[Stability for fixed $\ve > 0$]\label{th2b}
	Assume $\alpha >0 $ and  $\ve > 0$ fixed. Let $\vkseq \subset V$ converge to some $v^\delta \in V$ and let 
	\[u_k \in \arg\min \left\{ \Jevk(u) : u \in \D \right\}.\]
	Then, there exist a subsequence $\ukjseq$ which converges weakly to a minimizer $\uade$ of the functional $\Jevd$. Moreover, we have that
	$\Rqe(u_{k_j})\rightarrow \Rqe(\uade).$
\end{theorem}

\begin{proof}
	Since $u_k$ is a sequence of minimizers of $\Jevk$, it holds that
	$\Jevk(\uk) \leq \Jevk(u)$ for any $u \in \D$. 
	From Lemma \ref{lemma6}, there exists a subsequence $\ukjseq$ weakly converging to some $\ut \in \dom(F)$ such that $F(\ukj) \wconv F(\ut)$. Moreover, from the weak lower semi-continuity of $\norm{\cdot}^p_V$ and $\Rqe$ there holds
	\begin{eqnarray}
	\fl	\qquad \norm{F(\ut) - \vd}^p_V \leq \liminfj \norm{F(\ukj) - \vkj}^p_V \quad \text{ and } \quad \Rqe(\ut)\leq\liminfj \Rqe(\ukj). 
	\end{eqnarray} 
	Combining the above, we get
	\begin{eqnarray}
		\Jevd(\ut) 
		&\leq \liminfj \norm{F(\ukj) - \vkj}^p_V + \alpha \liminfj \Rqe(\ukj) \nonumber \\
		&\leq \liminfj \left\{ \norm{F(\ukj) - \vkj}^p_V + \alpha \Rqe(\ukj) \right\} \nonumber \\
		&= \liminfj \Jevkj(\ukj). \label{ineqAa}
	\end{eqnarray}
	On the other hand, for any $u \in \D$, we see that
	\begin{equation} \label{ineqBb} 
		\Jevd(u) = \limk \Jevk(u) 
		\geq \limsupj \Jevkj(\ukj)
		\geq \liminfj \Jevkj(\ukj).
	\end{equation}
	From \eref{ineqAa} and \eref{ineqBb} we conclude that $\Jevd(\ut) \leq \Jevd(u)$ for any $u \in \D$, that is, $\uade := \ut$ is a minimizer of $\Jevd$. Moreover, the weak lower semi-continuity of $\norm{\cdot}^p_V$ and $\Rqe$ implies that $\Rqe(\ukj) \to \Rqe(\uade)$.
\end{proof}

\begin{remark}
    In~\cite[Proposition 6]{Grasmair2008}, the authors additionally to $\Rq(\ukj) \to \Rq(\uad)$ prove that $\Rq(\ukj - \uad) \to 0$ for their functional $\Rq$. In our case, such a result cannot be inferred as weak convergence is not preserved under the nonlinearity of $d_\ve$. That is, assuming $\uk \wconv \uade$ we cannot prove that $\Rqe(\ukj - \uade) \to 0$. In order to obtain norm convergence, one can further assume $d_\ve(\uk) \wconv d_\ve(\uade)$. However, we choose not to make this additional assumption as it is quite restrictive.
\end{remark}

\begin{theorem}[Weak convergence for fixed $\ve > 0$] \label{ConvergenceFixedve}
	Let $\ve >0$ be fixed. Assume that $F(u)=v$ attains a solution in $\dom(\Rqe)$ and that $\alpha: (0,\infty) \to (0,\infty)$ 
	satisfies \[\alpha(\delta) \to 0 \mathrm{ ~ and ~ } \frac{\delta^p}{\alpha(\delta)} \to 0, \mathrm{ ~ as ~ } \delta \to 0.\] 
	Let $\delta_k \to 0$ and let $\vk \in V$ satisfy $\norm{v - v_k} \leq \delta_k$.
	Moreover, let $\alpha_k := \alpha(\delta_k)$ and \[u_k \in \arg\min \left\{\Jakevk(u): u \in \D\right\}.\] 
	Then, there exist an $\Rqe$-minimizing solution $\ud$ of $F(u)=v$ and a subsequence $\ukjseq$ with $\Rqe(\ukj) \to \Rqe(\ud)$. 
\end{theorem}

\begin{proof}
	Let $\ut \in \dom(\Rqe)$ be any solution of $F(\ut) = v$. 
	From the definition of $\uk$ it follows that
	\begin{eqnarray*}
		\Jevk(\uk) &= \norm{F(\uk) - \vk}^p_V + \alpha_k \Rqe(\uk) 
		\leq \delta_k^p + \alpha_k \Rqe(\ut).
	\end{eqnarray*}
	It can be easily seen that $\norm{F(\uk) - \vk}^p_V \leq \Jevk(\uk) \leq \delta_k^p + \alpha_k \Rqe(\ut)$ and together with the assumptions on $\alpha_k$ and $\delta_k$, we conclude that \(\norm{F(\uk) - \vk}_V \to 0\). For the penalty term we have $\Rqe(\uk) \leq \frac{\delta_k^p}{\alpha_k } + \Rqe(\ut)$ which yields
	\begin{equation}\label{eqconv2}
	 	\limsupk \Rqe(\uk) \leq  \Rqe(\ut), 
	\end{equation}
	when using the definition of the limit superior. Let $\alpha_{\max} := \max \{\alpha_k : k \in \N\}$, from the previous inequality there exists $M >0$ such that
	\[\limsupk \left\{\norm{F(\uk) - \vk}^p_V + \alpha_{\max} \Rqe(\uk)\right\} \leq M < \infty, ~ \forall k \in \N.\]
	Therefore, Lemma~\ref{lemma6} guarantees the existence of a subsequence $\ukjseq$ and some $\ud \in \dom(F)$ such that
	$\ukj \wconv \ud$ and $F(\ukj) \wconv F(\ud)$. 
	Since 
	\begin{eqnarray*}
		\fl \qquad \norm{F(\ukj) - v}_V = \norm{F(\ukj) - \vkj + \vkj - v}_V
		\leq \norm{F(\ukj) - \vkj}_V + \norm{\vkj - v}_V \to 0,
	\end{eqnarray*}
	it follows that $\norm{F(\ud) - v}_V = 0$, i.e., $F(\ud) = v$. 
	From the weak lower semi-continuity of $\Rqe$ and the fact that \eref{eqconv2} holds for any $\ut \in \dom(\Rqe)$ solving $F(\ut) = v$, we conclude 
	\begin{eqnarray*}
		\Rqe(\ud) \leq \liminfj \Rqe(\ukj) \leq \limsupj \Rqe(\ukj) \leq \Rqe(\ut).
	\end{eqnarray*}
	This shows that $\ud$ is an $\Rqe$-minimizing solution of $F(u) = v$ and $\Rqe(\ukj) \to \Rqe(\ud)$. 
\end{proof}


\subsection{Stability and convergence for vanishing tolerances}

In the previous results we always assumed a positive constant $\ve$. In this section, we consider a nonnegative sequence $\vek$, such that $\vek \to 0$. When the limit point of $\vek$ is $0$, we observe that $d_0(u) = \abs{u}$ gives 
\begin{equation}
    \Rqo(u) 
    = \int_\Omega \abs{u(x)}^q \dx =: \mathcal{R}_q(u). \label{ZeroEpsilonYieldsClassicalTK}
\end{equation}
Therefore, we obtain minimizers of the generalized Tikhonov functional. For that reason, the minimizer of $\Jvd(u) := \Jovd(u)$ is denoted by $\uad :=\uado$.

\begin{theorem}[Stability for $\vek \to 0$]\label{th41}
	Assume $\alpha >0 $. Let $\vkseq$ converge to $\vd \in V$, $\{\vek\}_{k \in \N}$ be a tolerance sequence converging to $0$ and let 
	\[ u_k \in \arg\min \left\{ \Jekvk(u) : u \in \D \right\}.\]
	Then, there exist $\{(\vekj, \ukj)\}_{j \in \N}$ and a minimizer $\uad$ of the functional $\Jvd$ such that $\Rq(u_{k_j}-\uad) \to 0$.
\end{theorem}

\begin{proof}
	The minimizing property of $\uk$ gives that
    \(\Jekvk(\uk) \leq \Jekvk(u), ~\forall u \in \D \).
	Lemma~\ref{lemma6}, guarantees the existence of a subsequence of $\uk$, denoted by $\ukjseq$, which converges to some $\ut \in \dom(F)$ and is such that $F(\ukj) \wconv F(\ut)$.
	From the weak lower semi-continuity of $\norm{\cdot}^p_V$ and $\Rqek$ and the fact that $\vek \to 0$, we have that
	\begin{eqnarray}\label{inequalityA}
	\fl \qquad	\Jovd(\ut) \leq \liminfj \norm{F(\ukj) - \vkj}^p_V + \alpha \liminfj \Rqekj(\ukj) \leq \liminfj \Jekjvkj(\ukj). 
	\end{eqnarray}
	On the other hand, since $\vek \to 0$, for any $u \in \dom(F)$ we have 
	\begin{eqnarray*}
	    \fl \qquad \Jovd(u) = \limk \Jekvk(u) \geq \limsupj \Jekjvkj(\ukj) \geq \liminfj \Jekjvkj(\ukj)	\overset{\eref{inequalityA}}{\geq} \Jovd(\ut). 
	\end{eqnarray*}
	Hence, based on the notation in~\eref{ZeroEpsilonYieldsClassicalTK}, we obtain \[\Jvd(u) =: \Jovd(u) \geq \Jovd(\ut) := \Jvd(\ut),\] 
	for all $u \in \dom(F)$, implying that $\uad := \ut$ is a minimizer of $\Jvd$. Moreover, $\Jvd(\ukj) \to \Jvd(\uad)$ and due to the fact that both $\norm{\cdot}^p$ and $\Rq$ are weakly lower semi-continuous, it follows that $\Rq(\ukj) \to \Rq(\uad)$. Then, with the use of~\cite[Lemma 2]{Grasmair2008} we conclude that $\Rq(\ukj - \ud) \to 0.$
\end{proof}

\begin{theorem}[Convergence for $\vek \to 0$] \label{Convergenceth5}
	Let $\{\vek\}_{k \in \N}$ be a tolerance sequence converging to $0$. We assume that $F(u)=v$ attains a solution in $\dom(\Rqek)$ and that $\alpha: (0,\infty) \to (0,\infty)$ 
	satisfies \[\alpha(\delta) \to 0 \text{ and } \frac{\delta^p}{\alpha(\delta)} \to 0, \text{ as } \delta \to 0.\] 
	Let $\delta_k \to 0$ and let $\vk \in V$ satisfy $\norm{v - v_k} \leq \delta_k$.
	Moreover, let $\alpha_k = \alpha(\delta_k)$ and \[u_k \in \arg\min \left\{\Jakekvk(u): u \in \D\right\}.\] 
	Then, there exist an $\Rq$-minimizing solution $\ud$ of $F(u)=v$ and a subsequence $\ukjseq$ with $\Rq(\ukj - \ud) \to 0$. 
\end{theorem}

\begin{proof}
	Let $\ut \in \dom(\Rqek)$ be any solution of $F(\ut) = v$. The minimizing property of $\uk$ implies
	\begin{eqnarray*}
		\Jakekvk (\uk) \leq \Jakekvk(\ut) &= \norm{F(\ut) - \vk }^p_V + \alpha_k \Rqek(\ut) \\ 
		&=\norm{v - \vk }^p_V + \alpha_k \Rqek(\ut) \\ 
		&\leq \delta_k^p + \alpha_k \Rqek(\ut). 
	\end{eqnarray*}
	Therefore, it follows that \(\norm{F(\uk) - \vk}^p_V \leq \Jakekvk (\uk) \leq  \delta_k^p + \alpha_k \Rqek(\ut)\). Then, taking the limit for $k \to \infty$ yields $\norm{F(\uk) - \vk}^p_V \to 0$ since we assumed that $\alpha_k \to 0$ and $\delta_k \to 0$ as $k \to \infty$. In a similar way, for the penalty term we have 
	\[\alpha_k \Rqek(\uk) \leq \Jakekvk (\uk) \leq  \delta_k^p + \alpha_k \Rqek(\ut),\] that is $\Rqek(\uk) \leq \frac{\delta_k^p}{\alpha_k } + \Rqek(\ut)$. 
	Taking the limit superior as $k \to \infty$ we obtain 
	\begin{equation}
		\limsupk \Rqek(\uk) \leq \limsupk \left\{\frac{\delta_k^p}{\alpha_k } + \Rqek(\ut) \right\} = \Rqo(\ut), \label{inequalityB}
	\end{equation} which is true for any solution $\ut$ of $F(\ut) = v$. \\
	With $\alpha_1 := \max\{\alpha_k : k \in \N\}$ and the previous calculation, there exists a constant $M >0$ such that  
	\[\limsupk \left\{\norm{F(\uk) - \vk}^p_V + \alpha_1 \Rqek(\uk) \right\} \leq M < \infty, ~ \forall k \in \N.\]
	From Lemma \ref{lemma6}, there exists a subsequence $\ukjseq$ weakly convergent to some $\ud \in \dom(F)$ such that $F(\ukj) \wconv F(\ud)$. Since 
	\begin{eqnarray}
	    \norm{F(\ukj) - v}^p &= \norm{F(\ukj) - \vkj + \vkj - v}^p \nonumber \\
	    &\leq 2^{p-1} \left(\norm{F(\ukj) - \vkj}^p + \norm{\vkj - v}^p \right) \to 0 
	\end{eqnarray}
	it follows that $F(\ud) = v.$ \newline
	From the weak lower semi-continuity of $\Rqek$, the fact that $\vek \to 0$ and \eref{inequalityB}, we obtain that
	\[\Rqo(\ud) \leq \liminfj \Rqekj(\ukj) \leq \limsupj \Rqekj(\ukj) \leq \Rqo(\ut),\] for all $\ut$ such that $F(\ut) = v$.
	Using the notation in \eref{ZeroEpsilonYieldsClassicalTK}, we conclude that $\Rq(\ud) \leq \Rq(\ut)$, for all $\ut$ such that $ F(\ut) = v.$ Hence, $\ud$ is an $\Rq$-minimizing solution of $F(u) = v$. Due to $\ukj \wconv \ud$ and the fact that $\Rq(\ukj) \to \Rq(\ud)$, and using~\cite[Lemma 2]{Grasmair2008}, we further conclude that $\Rq(\ukj - \ud) \to 0.$
\end{proof}

\section{Convergence rates}

In this section we present results on the convergence rates of minimizers of the functional~\eref{TikhonovRe}. Since we assume the parameter space to be a Banach space, we adopt the standard approach in Banach space settings and use the Bregman distance to estimate the difference between the regularized solution $\uade$ and the ground truth $\ud$. Some standard results on convergence rates are found in~\cite{Burger_2004,Engl_1989,Grasmair2008,JinMaass,Lorenz_2008}, while in~\cite{Grasmair_2010,Hofmann_et_al_2007,Schuster2012} exist convergence rates results using the Bregman distance. Moreover, for estimating the distance between $F(\uade)$ and $\vd$, we use the usual norm of the Banach space $V$.

The definition of the Bregman distance for $\Rqe$ requires the subdifferential of the functional $\Rqe: L_q(\Omega) \to \R$ at an element $u\in L_q(\Omega)$, which is given by
\begin{equation*}
 \fl \qquad \partial \Rqe(u) := \left\{ z \in L_q(\Omega)^* ~ : ~ \forall w \in L_q(\Omega) \quad \Rqe(w) \geq \Rqe(u) + {\inner{z,w-u}}_{{L_q(\Omega)}^*\times L_q(\Omega)}\right\},
\end{equation*}
where $L_q(\Omega)^*$ denotes the dual space of $L_q(\Omega)$ and  ${\inner{\cdot,\cdot}}_{{L_q(\Omega)}^*\times L_q(\Omega)}$ the dual pairing between  $L_q(\Omega)^*$ and $L_q(\Omega)$. 
Particularly for (finite) $n$-dimensional problems, like the numerical example presented in the next section, the $L_{q,\ve}$-insensitive measure appearing in the regularization functional is defined by $\norm{u}_{q,\ve}  = \norm{\de(u)}_q = {\left(\sum_{i=1}^{n} \abs{\de(u_i)}^q\right)}^{1/q} $. Using the classical subdifferential rules, for $q=1$ we compute the subdifferential
\begin{equation}
    \partial \mathcal{R}_{1,\ve} (u) 
    = \partial \norm{\de(u)}_1 = \partial \sum_{i=1}^{n}  \abs{\de(u_i)} = \sum_{i=1}^{n} \partial \abs{\de(u_i)} 
\end{equation}
with $i$-th sum component given by
\begin{equation}\label{Eq:Subdiff_Q1}
   \partial \abs{\de(u_i)} =
   \cases{\{-1\} &if $u_i <-\ve$ \\ 
   [-1,0] &if $u_i = -\ve$ \\ 
   \{0\} &if $\abs{u_i}<\ve$ \\ 
   [0,1] &if $u_i = \ve$ \\ 
   \{1\} &if $u_i> \ve$}.
\end{equation}
Similarly, for $q=2$ we have
\begin{equation}
    \partial \mathcal{R}_{2,\ve} (u) =\partial \norm{\de(u)}_2^2 
    = \partial \sum_{i=1}^n \abs{\de(u_i)}^2
    = \sum_{i=1}^n \partial \abs{\de(u_i)}^2
\end{equation}
with $i$-th sum component computed as
\begin{equation}
      \partial \abs{\de(u_i)}^2 = 
     \cases{u_i + \ve &if $u_i<-\ve$ \\ 
     0 &if $\abs{u_i}\leq \ve$ \\
     u_i - \ve &if $u_i>\ve$}.
\end{equation}
Note that the tolerance function is applied in a component wise sense for computing the above subdifferentials. 
The previous computations are confirmed in the subdifferential's formula for $1 \leq q \leq 2$
\begin{equation}
    \partial \Rqe(u) 
    = \sum_{i=1}^n \partial \abs{\de(u_i)}^q = q \abs{\de(u_i)}^{q-1} \partial \abs{\de(u_i)}
\end{equation}
where $\partial \abs{\de(u_i)}$ is determined by~\eref{Eq:Subdiff_Q1}.

It is worth noting that if the tolerance $\ve$ is not scalar but it is given as a vector with positive entries, then instead of $\de(u_i)$ there will be $d_{\ve_i}(u_i) = \max\{\abs{u_i} - \ve_i, 0\}$ in all of the above calculations. Given the subdifferential of $\Rqe$, we proceed with the Bregman distance and the convergence rates.

\begin{definition}[Bregman distance]\label{BregmanForRqe}
    Let $\ve >0$. Also, let $\Rqe : L_q(\Omega) \to \R \cup \{\infty\}$ be a convex and proper functional with subdifferential $\partial \Rqe \subset {L_q(\Omega)}^*$. Considering an element $\xi \in \partial \Rqe(u)$, 
    the Bregman distance of $\Rqe$ at $u \in L_q(\Omega)$ is defined by 
    \begin{eqnarray}
       D_{\xi}^\ve(\tilde{u}, u) := \Rqe(\tilde{u}) - \Rqe(u) - {\inner{\xi,\tilde{u} - u}}_{{L_q(\Omega)}^*\times L_q(\Omega)},
    \end{eqnarray}
    for $\tilde{u} \in L_q(\Omega)$ and it is only defined in the Bregman domain \[\mathcal{D}_B^\ve(\Rqe) := \{u \in \dom(\Rqe) ~:~ \partial \Rqe(u) \neq \emptyset\}.\]
\end{definition}

For notational simplicity, we use the usual inner product notation $\inner{\cdot,\cdot}$ for the dual pairing. Since we work in Banach spaces, there should not be any confusion with the notation of inner products in Hilbert spaces. Moreover, when writing $\alpha \sim \delta^{s}$ for $\alpha: (0,\infty) \to (0,\infty)$ and $s>0$, we mean that there exist $C \geq c > 0$ and $\delta_0>0$ such that $c \delta^s \leq \alpha(\delta) \leq C \delta^s$ for $0 < \delta < \delta_0$. 

The classical process for proving convergence rates requires an additional assumption on the smoothness of $F$ (restriction of its nonlinearity), as well as a source condition (in~\cite{HofmannYamamoto2010,Scherzer2008} general source conditions are discussed) which allows the estimation of the duality pairing appearing in the Bregman distance. Both are included in the following assumption.

\begin{assumption}[Smoothness of $F$ and source condition]
\label{A:assumptionForRates} Assume that the following hold:
 	\begin{enumerate}
	    \item The operator $F$ is G{\^a}teaux differentiable at $\ud$ and $F'$ denotes its G{\^a}teaux derivative.
        \item There exists a constant $\gamma >0$, such that 
        \[\norm{F(u) - F(\ud) - F'(\ud)(u - \ud)} \leq \gamma D_{\xi}^\ve(u,\ud)\] for all $u \in \dom(F) \cap \mathcal{B}_\rho(\ud)$, with a sufficiently large $\rho$.
	    \item There exists $w \in V$, such that $\xi = {F'(\ud)}^* w$ with $\gamma \norm{w} <1$.
    \end{enumerate}
\end{assumption}

\begin{theorem}(Convergence rates)\label{ConvRatesTikhonovTol}
    Let $\ve > 0$, $1 \leq p,q \leq 2$. Moreover, we consider that Assumptions~\ref{AssumptionOnF} and ~\ref{A:assumptionForRates} hold. Assume noisy data $\vd \in V$ such that $\norm{v -\vd}_V \leq \delta$ and that there exists an $\Rqe$-minimizing solution $\ud$ of \eref{invpb}, in the Bregman domain $\mathcal{D}_B^\ve$. For the minimizer $\uade$ of \eref{TikhonovRe},
    we prove the following estimates:
    \begin{itemize}
        \item[] If $p=1$ and $\alpha \beta_2 < 1$,
        \[\norm{F(\uade) - \vd}_V\leq\frac{(1 + \alpha \beta_2) \delta}{1 - \alpha \beta_2} \quad \text{ and } \quad D_\xi^\ve(\uade,\ud)\leq\frac{\delta + \alpha \beta_2 \delta}{\alpha (1 - \beta_1)}.\]
        \item[] If $p>1$,
        \begin{eqnarray*}
            \norm{F(\uade) - \vd}^p_V&\leq p_* \delta^p + p_* \alpha \beta_2 \delta + {(\alpha \beta_2)}^{p_*} \quad \text{ and} \\
            \quad \quad D_\xi^\ve(\uade, \ud)&\leq\frac{\delta^p + \alpha \beta_2 \delta + {{(\alpha \beta_2)}^{p_*}}/{p_*}}{\alpha (1 - \beta_1)},
        \end{eqnarray*}
        with $p_*$ being the conjugate of $p$ such that $1/p+1/p_* = 1$.
    \end{itemize}
    Moreover, we have:
     \begin{enumerate}
        \item[] For $p=1$ and the choice $\alpha \sim \delta^{s}$ with fixed $0<s<1$
        \begin{eqnarray*}
            \norm{F(\uade) - \vd}_V  = \mathcal{O}(\delta) \quad \text{ and } \quad
            D_\xi^\ve(\uade,\ud) = \mathcal{O}(\delta^{1-s}).
        \end{eqnarray*}
        \item[] For $p>1$ and the choice $\alpha \sim \delta^{p-1}$
         \begin{eqnarray*}
            \norm{F(\uade) - \vd}_V  = \mathcal{O}(\delta) \quad \text{ and } \quad
            D_\xi^\ve(\uade,\ud) = \mathcal{O}(\delta).
        \end{eqnarray*}
    \end{enumerate}
\end{theorem}

\begin{proof}
    We start by comparing the functional values $\Jevd(\uade)$ and $\Jevd(\ud)$. From the minimizing property of $\uade$, we obtain
    \begin{eqnarray*} 
        \fl \qquad
        \norm{F(\uade) - \vd}^p_V + \alpha \norm{\uade - u^*}_{q,\ve}^q &\leq 
        \norm{F(\ud) - \vd}^p_V + \alpha \norm{\ud - u^*}_{q,\ve}^q \\ 
        &\leq \delta^p + \alpha \norm{\ud - u^*}_{q,\ve}^q.
    \end{eqnarray*}
    Then, by reordering and gathering terms we use the Bregman distance $D_\xi^\ve(\uade, \ud)$, which yields 
    \begin{eqnarray*} 
        \norm{F(\uade) - \vd}^p_V + \alpha D_\xi^\ve(\uade, \ud) \leq \delta^p - \alpha \inner{\xi, \uade - \ud}.
    \end{eqnarray*}
    In the next step we employ the source condition~{\it(iii)} of Assumption~\ref{A:assumptionForRates} for rewriting the last term, which results into
    \begin{eqnarray} \label{MainEstim}
        \norm{F(\uade) - \vd}^p_V + \alpha D_\xi^\ve(\uade, \ud) &\leq 
        \delta^p - \alpha \inner{w, F'(\ud)(\uade - \ud)}.
    \end{eqnarray}
    Now, we focus on the dual pairing of the last term, for which we have
    \begin{eqnarray*}
        -\inner{w, F'(\ud)(\uade - \ud)} &=& \inner{w,-F'(\ud)(\uade - \ud)} \nonumber \\
        &\leq& \norm{w} \norm{-F'(\ud)(\uade - \ud)}.
    \end{eqnarray*}
    Adding and subtracting $F(\ud)- F(\uade)$ inside the last term and using the triangle inequality, yields 
    \begin{eqnarray*}
        -\inner{w, F'(\ud)(\uade - \ud)} &\leq& \norm{w} \norm{F(\uade) - F(\ud) - F'(\ud)(\uade - \ud)}  
        \nonumber \\ 
        &&+ \norm{w} \norm{F(\ud) - F(\uade)}.
    \end{eqnarray*}
    Furthermore, we use the smoothness assumption of $F$ defined in~{\it(ii)} of Assumption~\ref{A:assumptionForRates} to write
    \begin{eqnarray*}
        -\inner{w, F'(\ud)(\uade - \ud)} 
        &\leq& \gamma \norm{w} D_\xi^\ve(\uade, \ud) + \norm{w} \norm{F(\ud) - F(\uade)}
    \end{eqnarray*}
    and by defining constants $\beta_1,\beta_2>0$ such that $\beta_1 := \gamma \norm{w} < 1$ and $\beta_2 := \norm{w}$, we further obtain
    \begin{eqnarray}\label{innerProductEstim}
        -\inner{w, F'(\ud)(\uade - \ud)} \leq \beta_1 D_\xi^\ve(\uade, \ud) + \beta_2 \norm{F(\ud) - F(\uade)}.
    \end{eqnarray}
    In addition, we can estimate the term $\norm{F(\ud) - F(\uade)}$. We add and subtract $\vd \in V$ and use the triangle inequality to conclude
    \begin{eqnarray}\label{TriangleEstim}
       \norm{F(\ud) - F(\uade)}_V 
       \leq \delta + \norm{F(\uade) - \vd}_V.
    \end{eqnarray}
    Substituting the estimates \eref{innerProductEstim}, \eref{TriangleEstim} into \eref{MainEstim}, we have
    \begin{eqnarray}\label{eq:PreEstimate}
        \fl \quad
        \norm{F(\uade) - \vd}^p_V + \alpha D_\xi^\ve(\uade, \ud) &\leq&  
        \frac{\delta^p}{p} + \alpha \beta_1 D_\xi^\ve(\uade, \ud) + \alpha \beta_2 \delta 
        + \alpha \beta_2 \norm{F(\uade) - \vd}_V.
    \end{eqnarray}
    For $p=1$, rearranging \eref{eq:PreEstimate} yields
    \begin{eqnarray*}
       (1 - \alpha \beta_2) \norm{F(\uade) - \vd}_V + \alpha (1 - \beta_1) D_\xi^\ve(\uade, \ud) \leq \delta + \alpha \beta_2 \delta.
    \end{eqnarray*}
    For sufficiently small $\alpha>0$ such that $\alpha \beta_2<1$, the first term is nonnegative. Moreover, the second term is nonnegative by assumption since $\beta_1 < 1$. Therefore, we can derive the following estimates
    \begin{eqnarray*}
        \norm{F(\uade) - \vd}_V \leq \frac{(1 + \alpha \beta_2) \delta}{1 - \alpha \beta_2} ~ \text{ and } ~  D_\xi^\ve(\uade,\ud) \leq \frac{\delta + \alpha \beta_2 \delta}{\alpha (1 - \beta_1)}.
    \end{eqnarray*}
    Choosing $\alpha \sim \delta^s$ with fixed $0 < s <1$, we obtain 
    \[ \norm{F(\uade) - \vd}_V  = \mathcal{O}(\delta) \quad \text{ and } \quad
       D_\xi^\ve(\uade,\ud) = \mathcal{O}(\delta^{1-s}).\]
    For $p > 1$, we have
    \begin{eqnarray*}
        \norm{F(\uade) - \vd}^p_V &-& \alpha \beta_2 \norm{F(\uade) - \vd}_V + \alpha D_\xi^\ve(\uade, \ud) \\
        &&\leq \delta^p + \alpha \beta_1 D_\xi^\ve(\uade, \ud) + \alpha \beta_2 \delta.
    \end{eqnarray*}
    Applying Young's inequality $ab \leq {a^p}/p + {b^{p_*}}/{p_*}, \text{ for } 1/p+1/{p_*} = 1$ with $a = \norm{F(\uade) - \vd}_V$ and $b = \alpha \beta_2 p$, yields
    \begin{eqnarray*}
      \fl \qquad 
      \left(1 - \frac{1}{p} \right)\norm{F(\uade) - \vd}^p_V + \alpha (1 - \beta_1) D_\xi^\ve(\uade, \ud) \leq \delta^p + \alpha \beta_2 \delta + \frac{{(\alpha \beta_2)}^{p_*}}{p_*}.
    \end{eqnarray*} 
    Both terms on the left hand side of the last inequality are nonnegative. Therefore, by neglecting the other nonnegative term, respectively, we conclude the following estimates
    \begin{eqnarray*}
        \norm{F(\uade) - \vd}^p_V 
        &\leq p_* \delta^p + p_* \alpha \beta_2 \delta + {(\alpha \beta_2)}^{p_*} \quad \text{ and} \\
        \qquad D_\xi^\ve(\uade, \ud)&\leq \frac{\delta^p + \alpha \beta_2 \delta + {{(\alpha \beta_2)}^{p_*}}/{p_*}}{\alpha (1 - \beta_1)}.
    \end{eqnarray*}
    The choice $\alpha \sim \delta^{p-1}$ yields    
     \(\norm{F(\uade) - \vd}_V  = \mathcal{O}(\delta) ~ \text{ and } ~ D_\xi^\ve(\uade,\ud) = \mathcal{O}(\delta)\).
\end{proof}

Except the convergence rates in the Bregman distance, one could also derive an estimate in the $L_{q,\ve}$-insensitive measure by using the inequality \eref{ineq1}. In the case of $q = 2$, i.e., in Hilbert space setting, we have the following remark. 

\begin{remark}
    For $q=2$ and classical penalty $\mathcal{R}_2(u) = \norm{u - u^*}^2_2$ we can transfer the estimates from the Bregman distance to the usual norm, see~\cite{Scherzer2008,JinMaass}. When including tolerances in the regularization term, such an equivalence does not hold true anymore. As previously stated, it is possible to use inequality~\eref{ineq1} to estimate 
    \[\norm{\uade - \ud}_{2,\ve}^2 \leq \norm{\uade - \ud}_2^2 = D_\xi(\uade, \ud),\]
    where the Bregman distance in the last expression is the one calculated for regularization without tolerances. However, an estimate relating $D_\xi^\ve(\uade,\ud)$ to $\norm{\uade - \ud}_{2,\ve}^2$ is not obvious. 
\end{remark}


\section{Numerical consideration}

For the numerical minimization of our functional, we use a subgradient algorithm introduced in~\cite[Algorithm 1]{Gralla2020}. The suggested method is a subgradient algorithm with adaptive decreasing step size. The authors in this article prove the stability and convergence of the algorithm and motivate its effectiveness by comparing their numerical results to those of other existing methods. In their examples the authors consider engineering applications for denoising and deblurring of 1D and 2D signals. For further details on the algorithm, we refer the reader to~\cite{Gralla2020}.


For illustrating the effect of tolerances in the solution when minimizing the altered Tikhonov functional, we present an example of noisy data differentiation. In the sequel, we consider the problem $Ku=v$ with linear integral operator $K : L_2(0,1) \to L_2(0,1)$ defined by 
\begin{equation}
    Ku(x) = \int_0^x u(s)~\ds, \quad u: [0,1] ~ \to ~ \mathbb{R}.
\end{equation}
We consider as reference solution the function $u^* = \sin(2\pi x)$ for $ x \in [0,1]$ and we wish to approximate the true solution $u^\dagger$ which is assumed to lie within the tolerance area around $u^*$ defined as 
\begin{equation}\label{ToleranceTube}
    \mathrm{Tube}_\ve(u^*) := u^* \pm \ve, \quad \text{ for } \ve >0.
\end{equation}
In our example we take $\ve$ to be a positive scalar but it could also be assumed as a nonnegative real function. The noisy data are then created as $v^\delta = Ku^\dagger + n(\delta)$, for a certain noise level $\delta >0$. In the results that follow we minimize the functional 
\begin{equation}\label{E:MinimizationFunctonial_Tol}
    \Jevd (u) = \frac{1}{2}\norm{Ku-\vd}^2 + \frac{\alpha}{q} \norm{u - u^*}^q_{q,\ve}, \quad q \in \{1,2\}, ~\ve \geq 0.
\end{equation}

We discretize the operator on the grid $x_i = (i - {\frac{1}{2}})h$, for $i = 1, \dots, N$ and $h = \frac{1}{N}$ with $N=600$ discretization points. This yields an $N\times N$ matrix (which, for simplicity, we denote again by $K$) with the following structure

\[K = h \left(
    \begin{array}{cccc}
        \frac{1}{2} & 0 & \ldots & 0 \\
        1 & \frac{1}{2} & \ddots & \vdots \\
        \vdots & \ddots & \ddots & 0 \\
        1 & \ldots & 1 & \frac{1}{2}  
    \end{array} 
    \right).\] 

In \Fref{F2a}, we show the true solution $\ud$, the reference solution $u^*$ and the tolerance area considered for $\ve = 0.3$. In \Fref{F2b}, we plot the true and noisy data created for $\delta = 0.001$. We compare the regularized solution $\uade$ obtained from the minimization of \eref{E:MinimizationFunctonial_Tol}, to the solution $\uad$ of the generalized Tikhonov functional. In \Fref{F3}, we compare the reconstructions $\uade, \uad$ to the true solution $\ud$ in the following two cases: for $\alpha = 0.001$ and $q=1$ (\Fref{F3a}) and for $\alpha = 0.01$ and $q = 2$ (\Fref{F3b}). In both cases $\uade$ has been computed with $\ve = 0.3$.

\begin{figure}
	\begin{center}		
		\subfloat[]
		{\label{F2a}
		\includegraphics[width=0.45\textwidth]{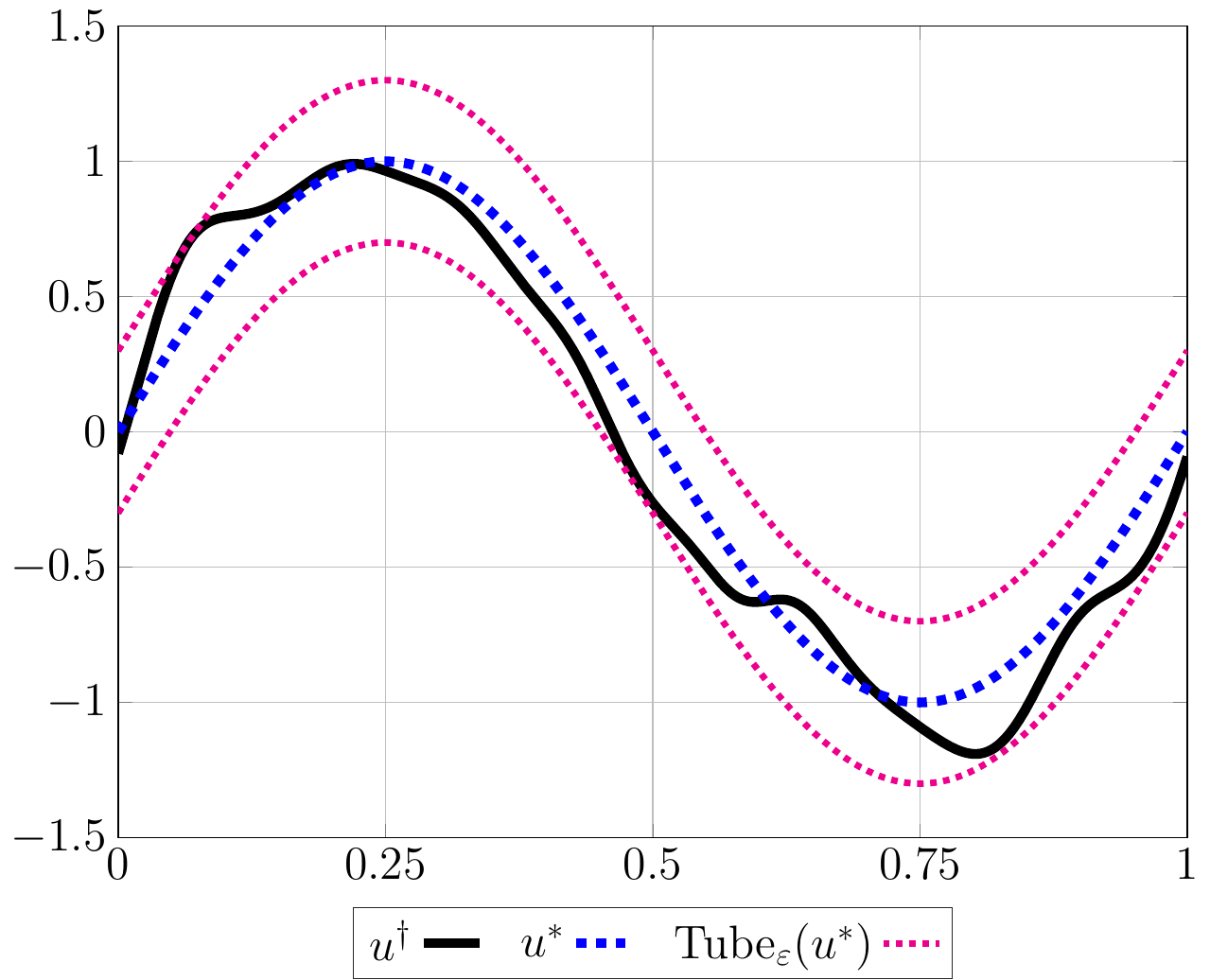}}
		\hspace{2pt}
		\subfloat[]
		{\label{F2b}
		\includegraphics[width=0.45\textwidth]{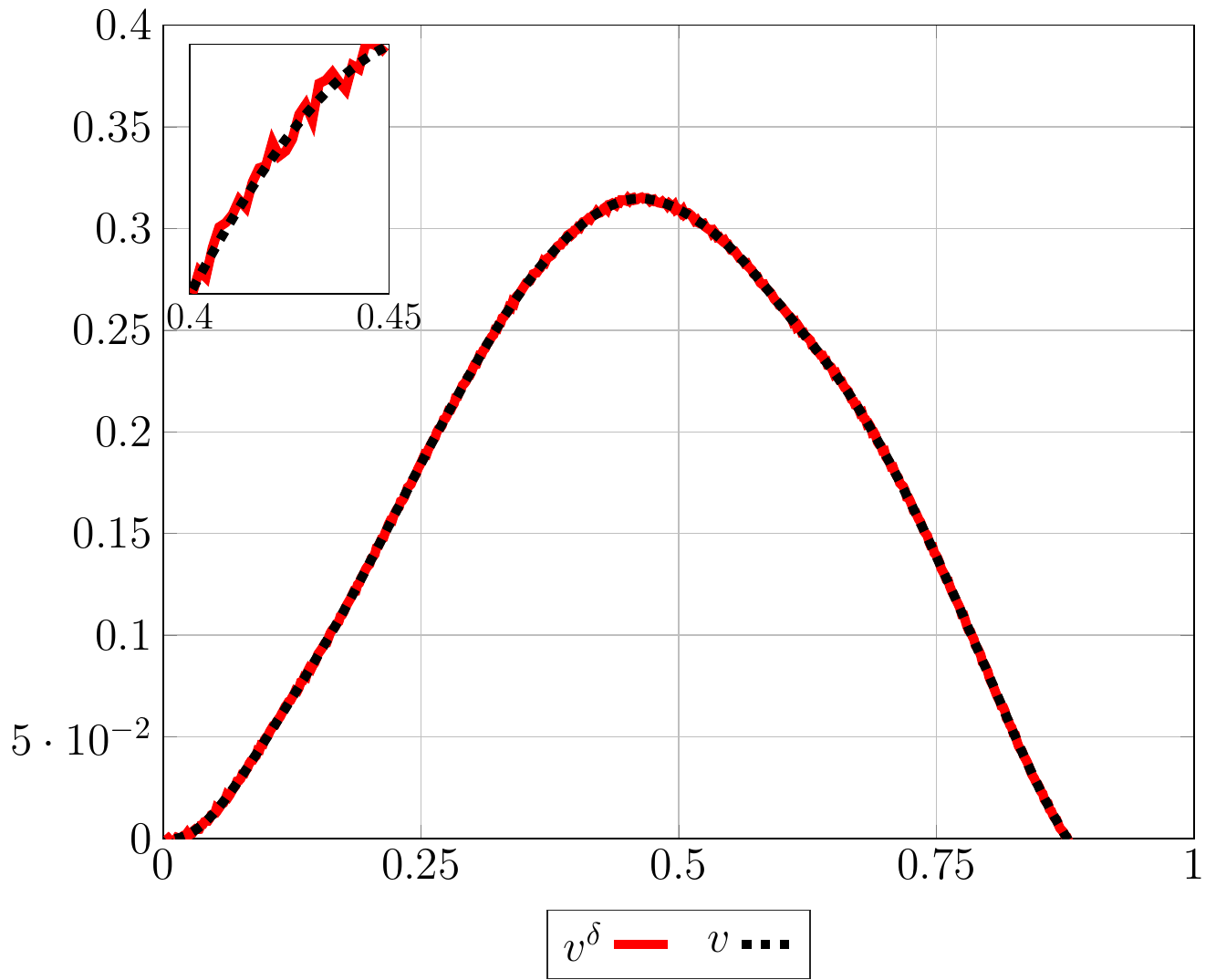}}
	\end{center}
	\caption{The true $\ud$ and reference solution $u^*$ together with the defined tolerance area for $\ve = 0.3$ are shown in (a). The true data $v = K u^\dagger$ and noisy data $\vd = v + n(\delta)$ created with $\delta = 0.001$ are shown in (b).}
\end{figure}

\begin{figure}
	\begin{center}
		\subfloat[]
		{\label{F3a}
		\includegraphics[width=0.45\textwidth]{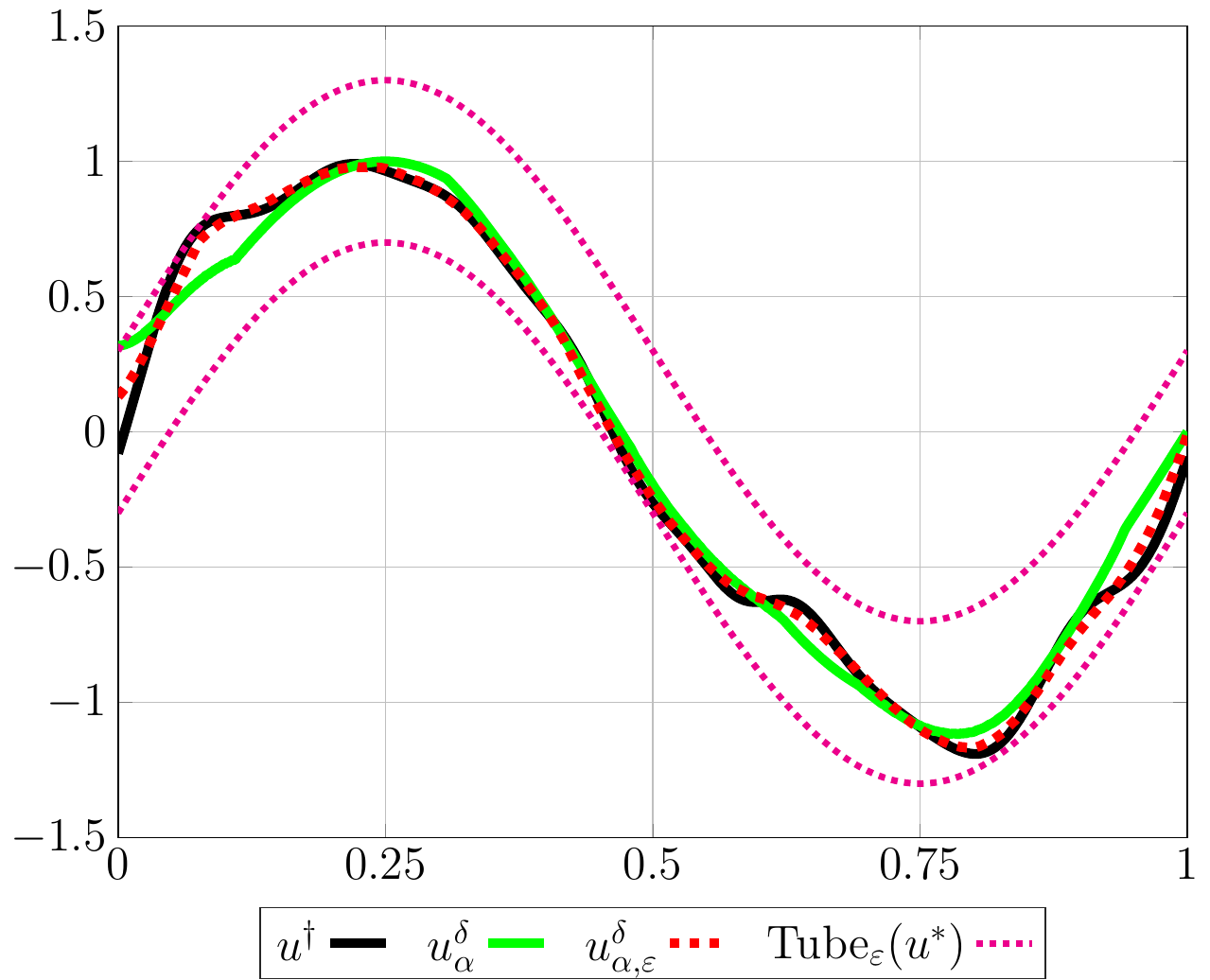}}
		\hspace{2pt}
		\subfloat[]
		{\label{F3b}
		\includegraphics[width=0.45\textwidth]{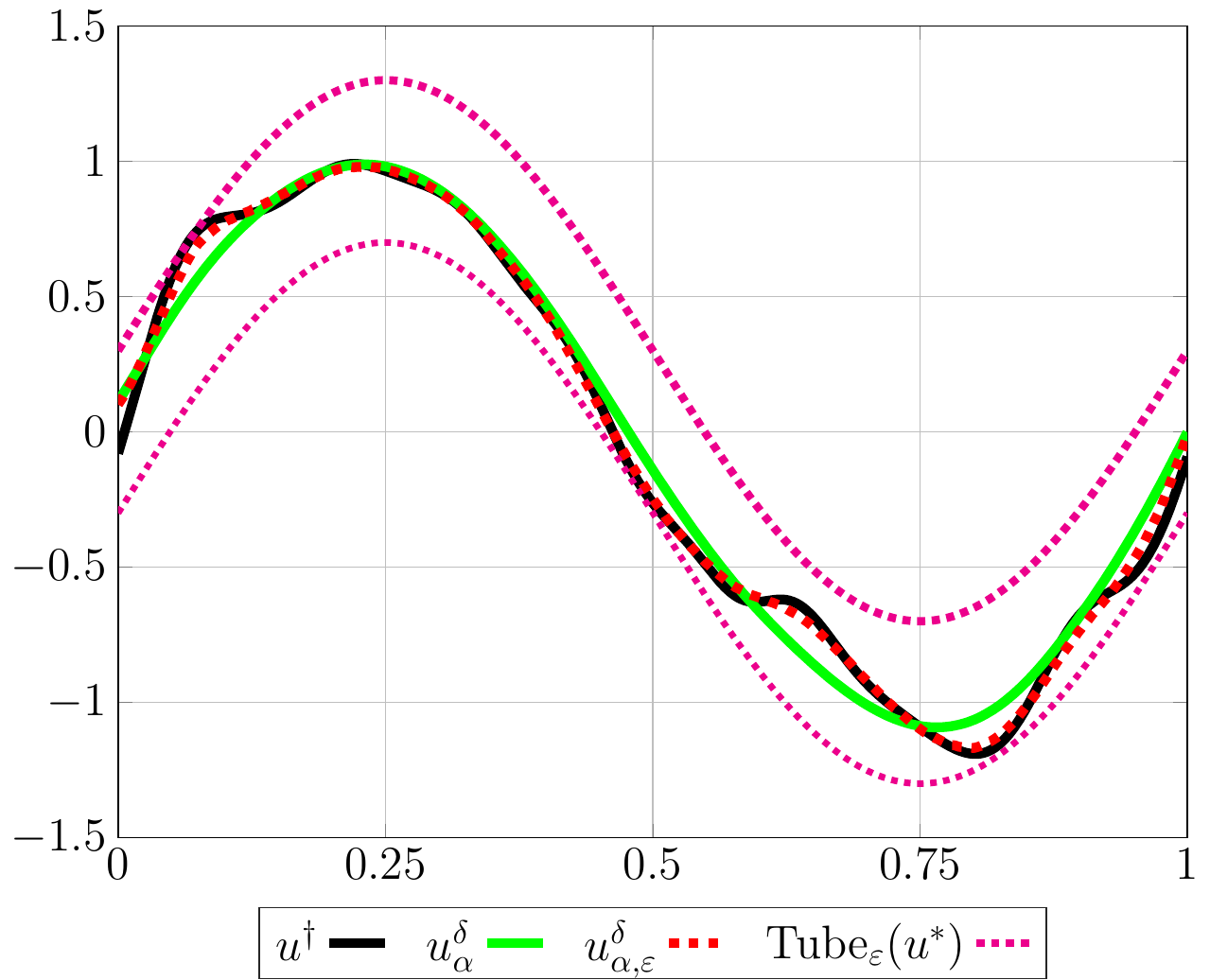}}
	\end{center}
	\caption{Comparison of the reconstructions $\uade$ and $\uad$ to the ground truth $\ud$. In both cases the chosen tolerance is $\ve = 0.3$. In (a) we use $\alpha = 0.001$ and $q=1$ and in (b) $\alpha =0.01$ and $q=2$.}
	\label{F3}
\end{figure}

\Fref{F3} shows that it is possible to obtain better reconstructions than those of the generalized Tikhonov regularization. Intuitively, the tolerances in the penalty term can be interpreted as further regularization of the solution. Their effect, however, depends on the choice of the regularization parameter as for a very small $\alpha$ the influence of tolerances can be insignificant.

\subsection{Error behavior}

\begin{wrapfigure}[19]{o}{0.45\textwidth}       
    \includegraphics[width=0.45\textwidth]{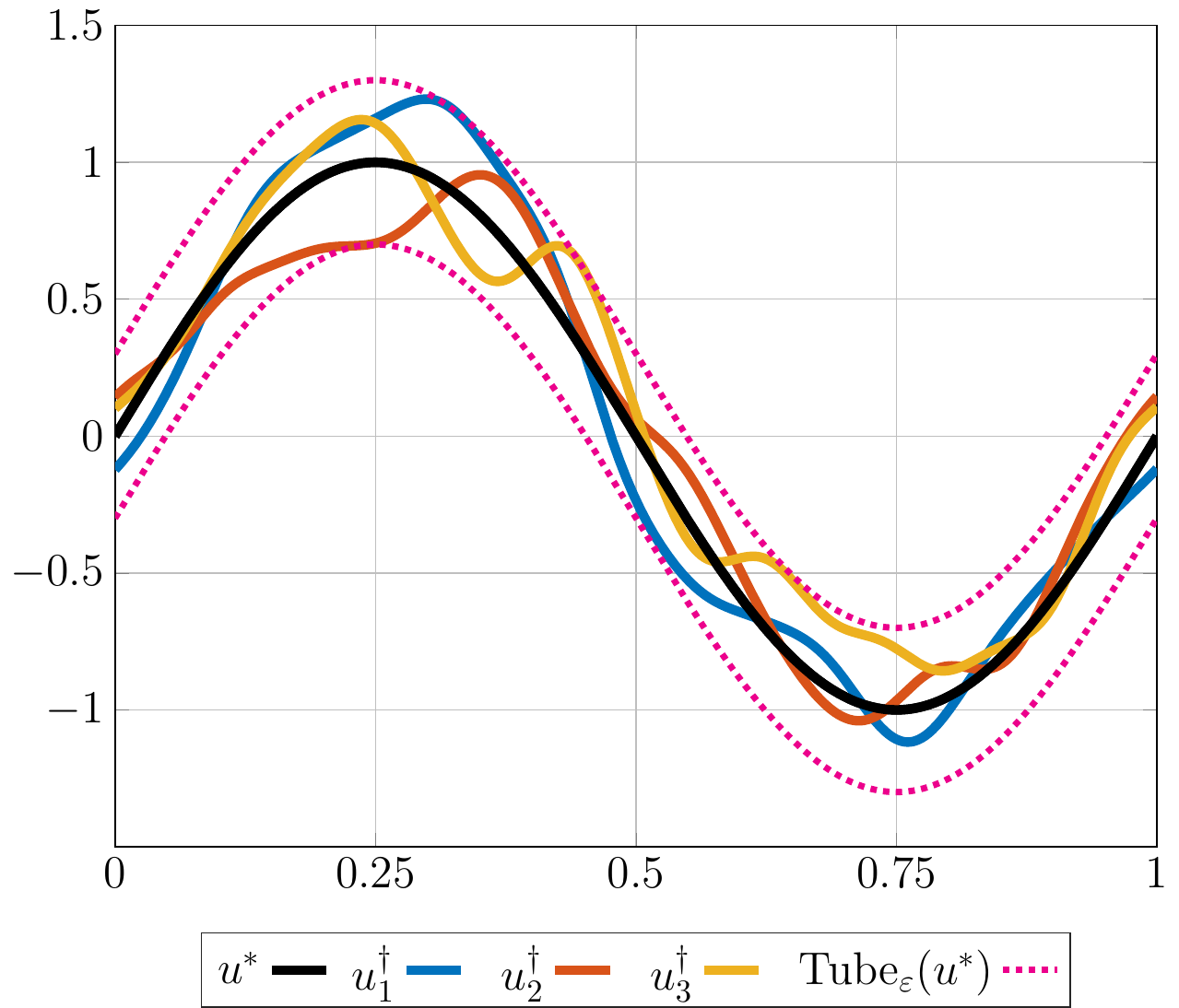} 
    \caption{Example of $u_i$ for $i= 1,2,3$ with $u^* = \sin(2\pi x)$ and $\ve = 0.3$.}
    \label{F:figure4}
\end{wrapfigure}
Now, we examine the behavior of the approximation error for different tolerances. We assume the same reference solution $u^*$ and we use various values of $\ve \in [0,1.2]$. For each value of $\ve$ we define the corresponding tolerance area $\mathrm{Tube}_\ve(u^*)$ as in~\eref{ToleranceTube} and we perform $N=50$ simulations over which we calculate the mean approximation error $\norm{\uade - \ud}_{q,\ve}$. Moreover, we denote the true solution by $u_i^\dagger$, where $i=1,\ldots,N$ is the index of the $i$-th run. Each $u_i^\dagger$ is generated as a random and smooth perturbation of $u^*$ inside the tolerance area $\mathrm{Tube}_\ve(u^*)$. Therefore, the ground truth and the noisy data are computed as
\begin{eqnarray}
    u_i^\dagger &= u^* + \ve \eta_i \in \mathrm{Tube}_\ve(u^*), \label{E:Smooth_perturbation} \\
    v_i^{\delta} &= Ku_i + n(\delta),
\end{eqnarray} 
for a random (but smooth) $\eta_i$ and $i = 1,\dots N$. The smooth perturbation within the tolerance area is created by convolution of a normally distributed, random vector with the zero-mean Gaussian distribution with standard deviation $\sigma=0.08$. 
This random vector is further weighted by the tolerance value $\ve$ and then added to $u^*$ as done in~\eref{E:Smooth_perturbation}.
Note that the noise level $\delta = 0.005$ and regularization parameter $\alpha = 0.001$ are kept the unchanged throughout all simulations as we only examine the resulting error for different $\ve \in [0,1.2]$. \Fref{F:figure4} is an illustration of the $u_i^\dagger$ created for $i = 1, 2, 3$. 

\begin{figure}
	\begin{center}		
		\subfloat[]{
		\includegraphics[width=0.45\textwidth]{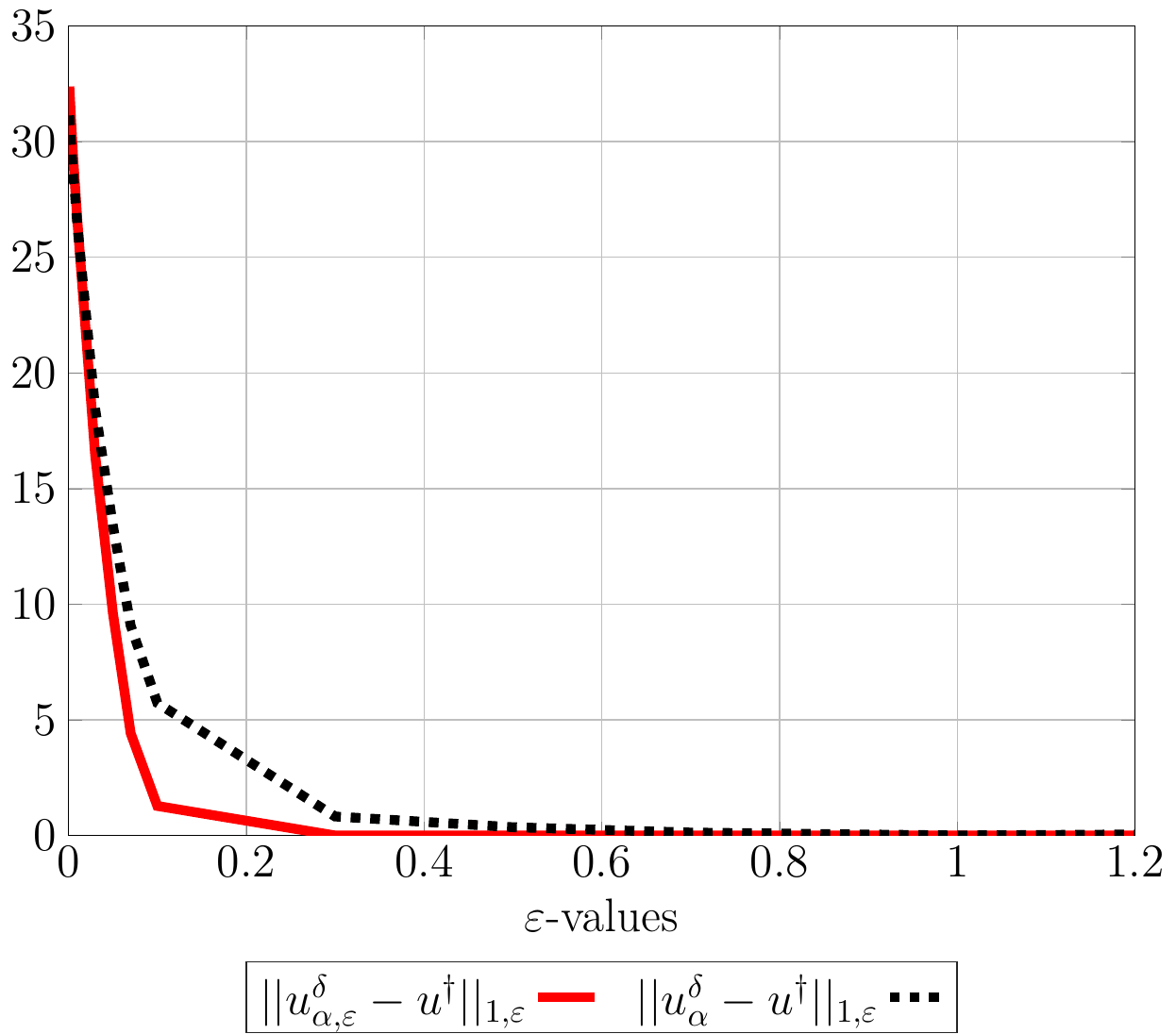}}
		\hspace{2pt}
		\subfloat[]{
		\includegraphics[width=0.45\textwidth]{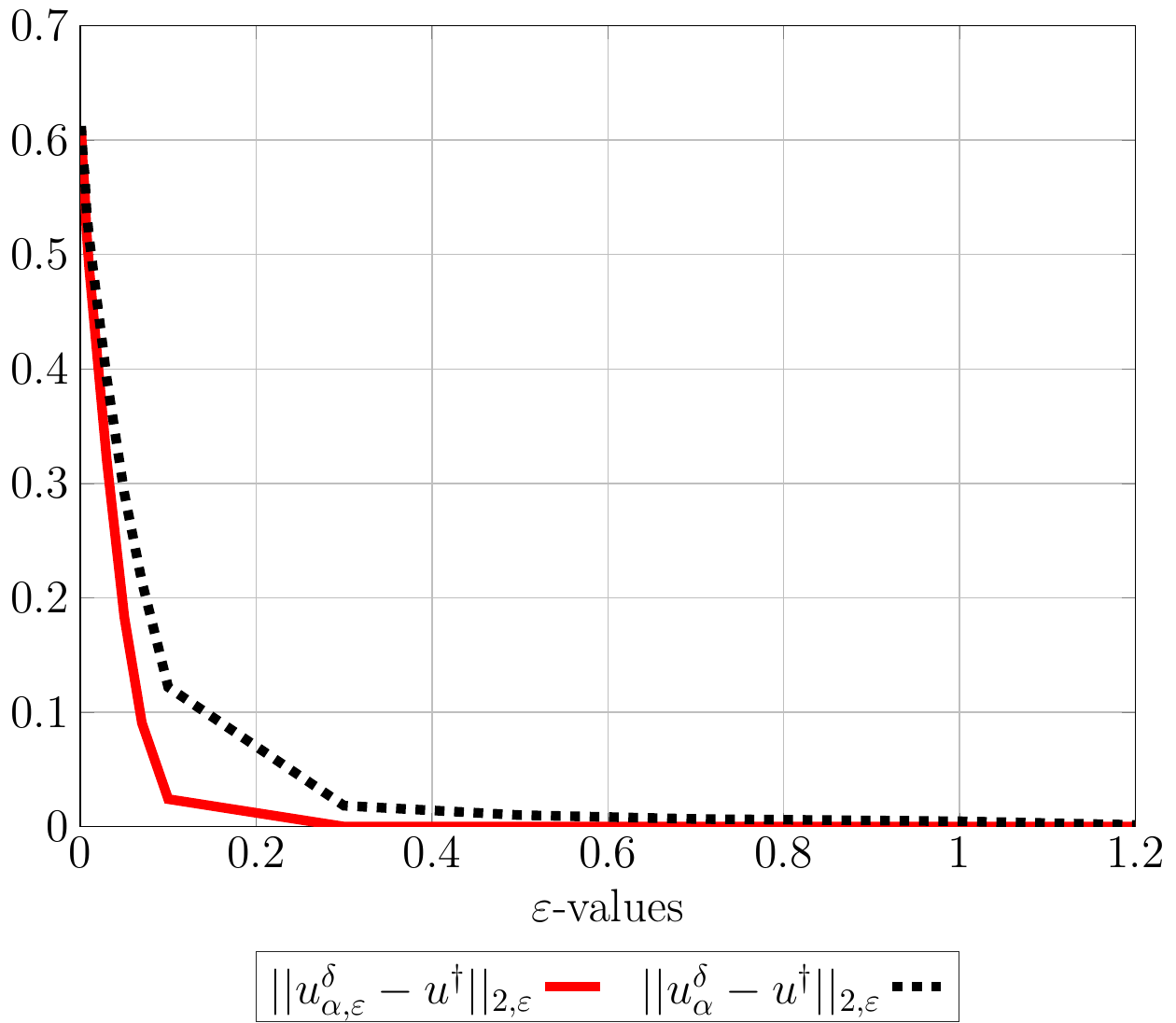}}
	\end{center}
	\caption{The mean approximation error computed over 50 runs. In both cases, we use regularization parameter $\alpha = 0.001$, noise level $\delta = 0.005$ and tolerance values $\ve \in [0,1.2]$. The error is computed using the $\ve$-insensitive measure for $q=1$ in (a) and $q=2$ in (b).}
	\label{F:figure5}
\end{figure}
\Fref{F:figure5} shows the error between the true solution $\ud$ and the reconstructions $\uad$, $\uade$ calculated in the $\ve$-insensitive measure for different values of $\ve$. In (a) we assume $q=1$ and in (b) $q=2$. Both plots reveal that the error obtained from our approach (red solid line) is smaller (or, at worst, equal) than the error calculated for the generalized Tikhonov minimizers (black dashed line). In addition, the use of the $\ve$-insensitive measure ensures that our reconstructions remain within the prescribed tolerance area.

\subsection{Choosing the regularization parameter}

The potential for obtaining better results when using tolerances takes us to the step of examining how we can enhance the quality of our reconstructions. An important task in Tikhonov regularization is the choice of the regularization parameter $\alpha>0$. Since we include tolerances in the regularization, we seek their effect on the solution and whether $\ve$ needs to be chosen according to the value of $\alpha$. In some applications an indication for the appropriate size of tolerances may exist, meaning that their value cannot be arbitrarily large. In that case, we only deal with finding the optimal value for $\alpha$ by using existing parameter choice strategies, see for instance~\cite{Anzengruber_Ramlau2009,Ito_Jin_Zou_2011}. However, when both $\alpha$ and $\ve$ need to be tuned, one can think of ways to combine them for improving the final solution.

The L-curve method is often used to gain an insight on the optimal value of $\alpha$, its use for the numerical solution of inverse problems is discussed in~\cite{Hansen2000}. It is created by plotting the discrepancy norm $\norm{K\uade - \vd}_2$ against the norm of the regularized solution $\norm{\uade-u^*}_{q,\ve}$ for different values of $\alpha$. The L-curve shows the trade-off between the fit to the given data and the size of the regularized solution, and the optimal value of $\alpha$ is found near the maximum curvature. 

In~\Fref{F:figure6} the L-curve for $q =1$ (left) and $q=2$ (right) for five different values of $\varepsilon$ and for $\alpha \in [10^{-12},1]$ is shown. In order to compare the L-curves in a similar scale, we assume $u^* = \sin(2\pi x)$ and the true solution is taken as a smooth perturbation inside the tolerance area for the fixed value $\ve^\dagger=0.5$. This is used in all simulations for different values of $\ve$. In this figure one can observe the different scaling (the L-curve is shifted down for larger $\ve$) due to the value of the corresponding tolerance. Moreover, when $\ve$ and $\alpha$ both become larger, the regularization term (y-axis of the L-curve) goes faster to $0$. In both cases the L-curve is not sharp as normally expected in the classical Tikhonov regularization. The nature of the tolerance function $\de$ indicates that the connection between $\alpha$ and $\ve$, and, in particular, their optimal combination, is not straightforward through the L-curve. Therefore, one has to look for more sophisticated parameter choice rules.

\begin{figure}
   \centering
    \includegraphics[width=0.9\textwidth]{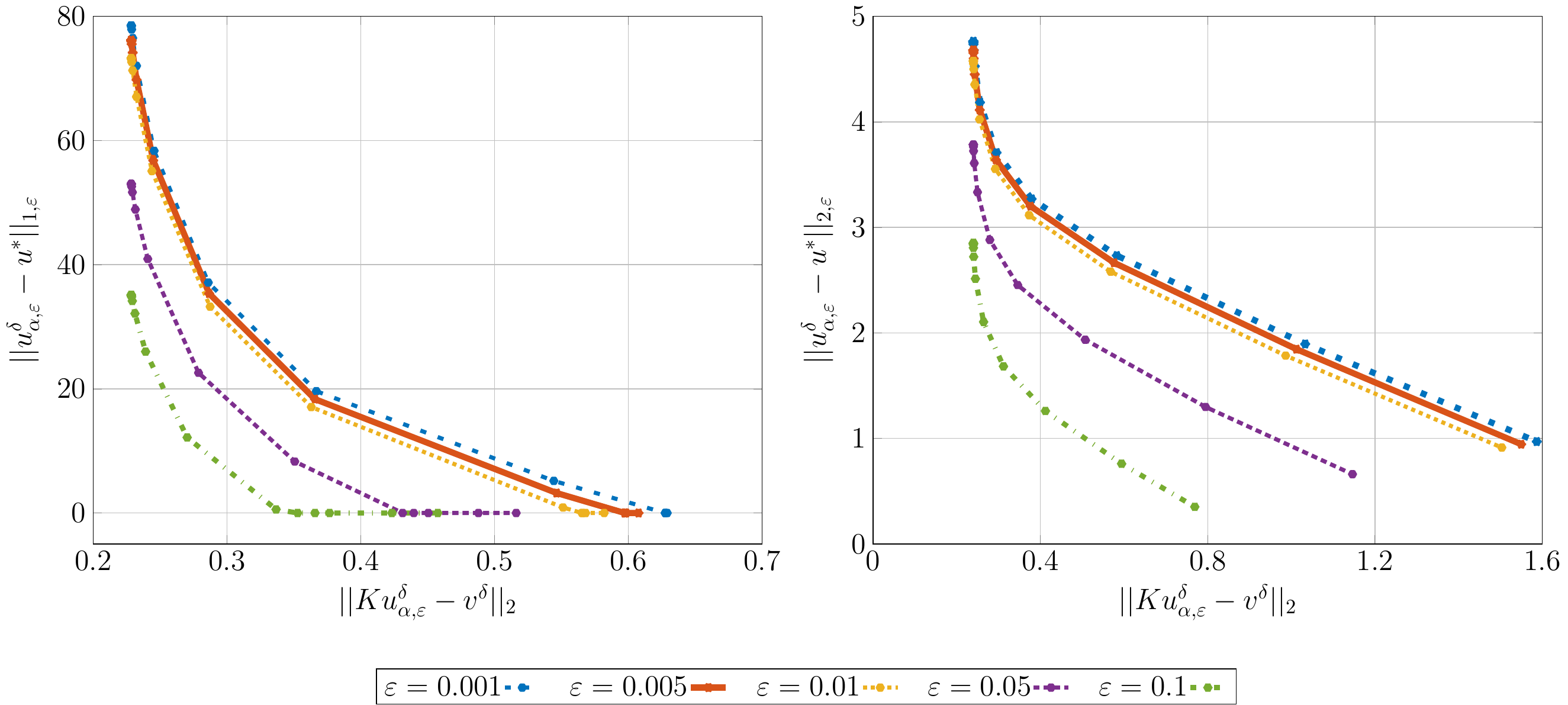}
    \caption{L-curve for different values of tolerance $\varepsilon$. On the left we have the results for $q=1$, on the right for $q=2$. In both cases we used $p=2$, $\delta = 0.015$, and $\alpha \in [10^{-12},1]$.}
    \label{F:figure6}
\end{figure}

A different method, which is based on the noise level in the data, is the so-called Morozov's discrepancy principle~\cite{Bonesky_2008,Morozov1966}. Given an estimate of the noise level, the idea of the discrepancy principle is to accept reconstructions which create measurements with the similar error as the one in the noisy data. This simply translates into choosing the maximum $\alpha>0$ such that 
\begin{equation}\label{E:discrepancy_principle}
    G(u_{{\alpha}}^\delta) := \norm{K u_\alpha^\delta - \vd} \leq \tau \delta, \quad \text{ for } ~ \uad := \arg\min \Jvd(u),
\end{equation}
with $\tau \geq1$ and an estimate of the noise level $\delta>0$. Here, we use the discrepancy principle for identifying the optimal regularization parameter $\alpha_{\mathrm{opt}}$ when minimizing the generalized Tikhonov functional $\Jvd$ (no tolerance assumption). Then, we use it as regularization parameter in our functional $\Jevd$. That is, we compare the optimal reconstruction of the generalized Tikhonov to the minimizer which we compute incorporating tolerances in the regularization. In the following figures we show these results for the discrepancy principle given as in \eref{E:discrepancy_principle} with $\tau = 2$, and we examine the cases $q = 1$ (in~\Fref{F:figure7}) and $q=2$ (in~\Fref{F:figure9}) for the penalty norm.

\begin{figure}
    \centering
    \includegraphics[width=0.7\textwidth]{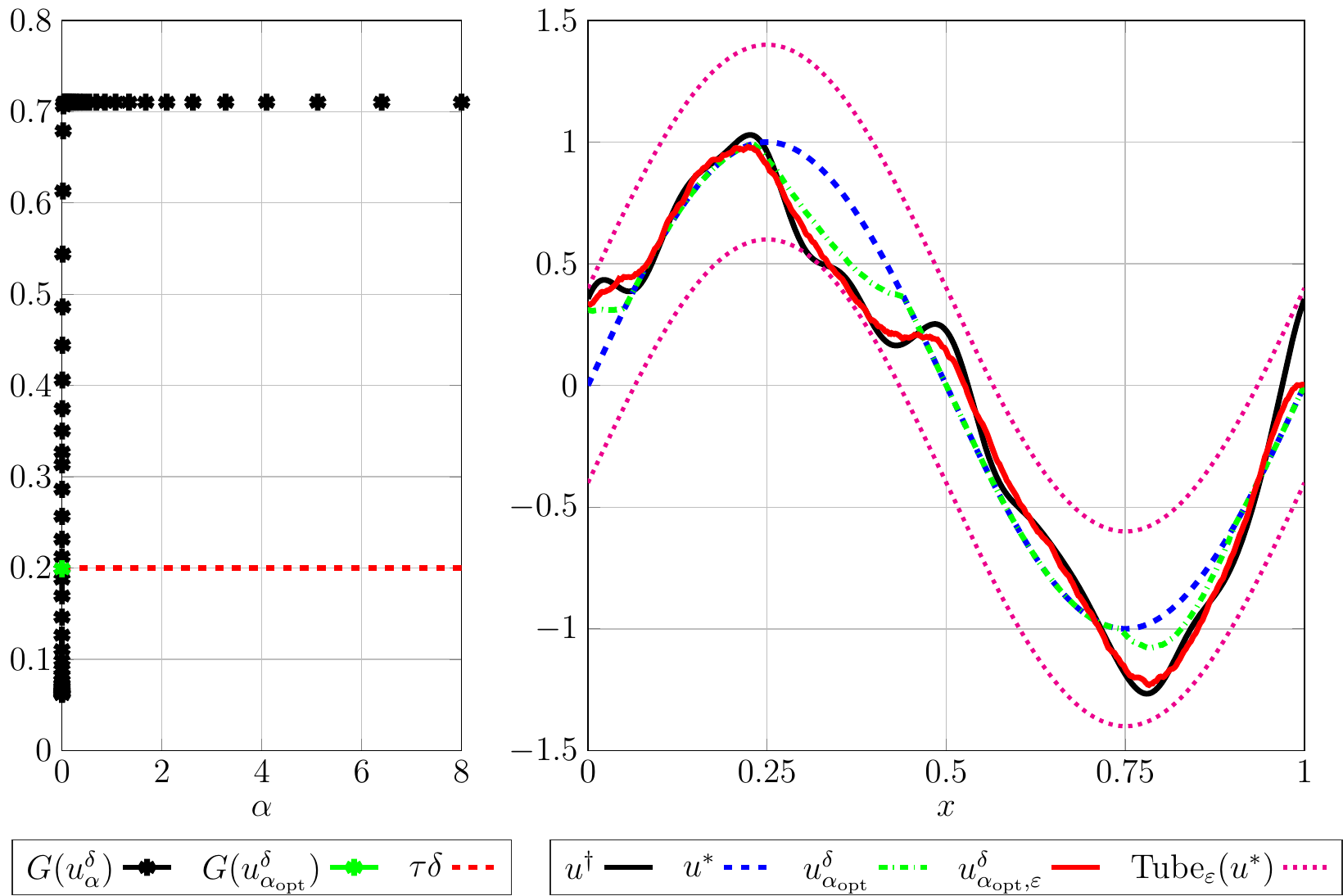}
    \caption{Morozov's discrepancy principle for $\delta=0.1$, $\tau=2$, $\ve=0.4$ and $q=1$. On the left, the value of $G(\uad)$ is plotted for different $\alpha$ and for the $\alpha_{\mathrm{opt}}=0.0011$ (in green). On the right, we plot the reconstructions $u_{\alpha_{\mathrm{opt}},\ve}^\delta$, $u_{\alpha_{\mathrm{opt}}}^\delta$ calculated with $\alpha_{\mathrm{opt}}$ and compare them to $\ud$ and $u^*$.}
    \label{F:figure7}
\end{figure}

\begin{figure}
    \centering
    \includegraphics[width=0.7\textwidth]{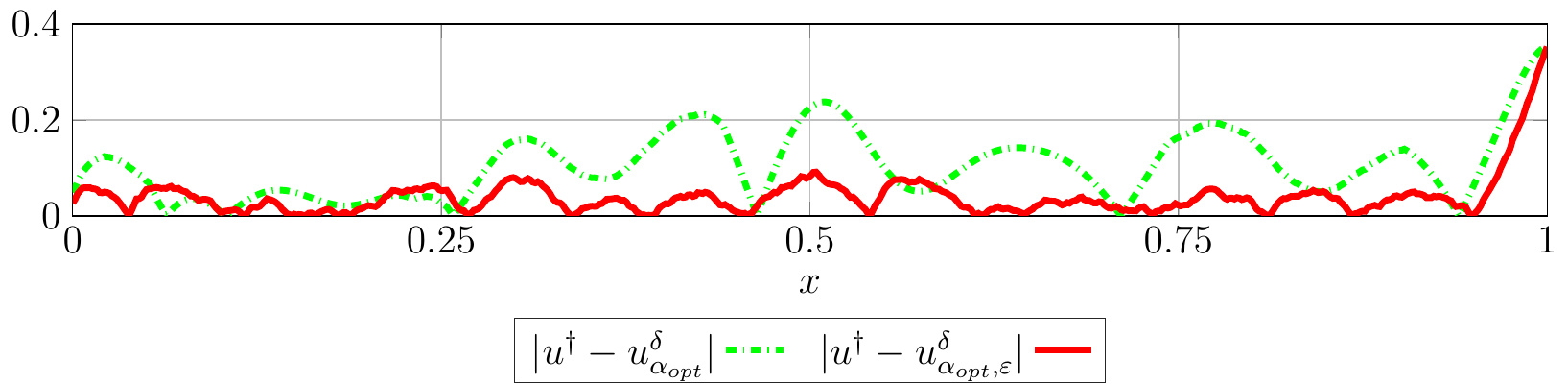}
    \caption{Absolute error of the reconstructions $u_{\alpha_{\mathrm{opt},\ve}}^\delta$ and $u_{\alpha_{\mathrm{opt}}}^\delta$ shown in~\Fref{F:figure7}.}
    \label{F:figure8}
\end{figure}

\begin{figure}
    \centering
    \includegraphics[width=0.7\textwidth]{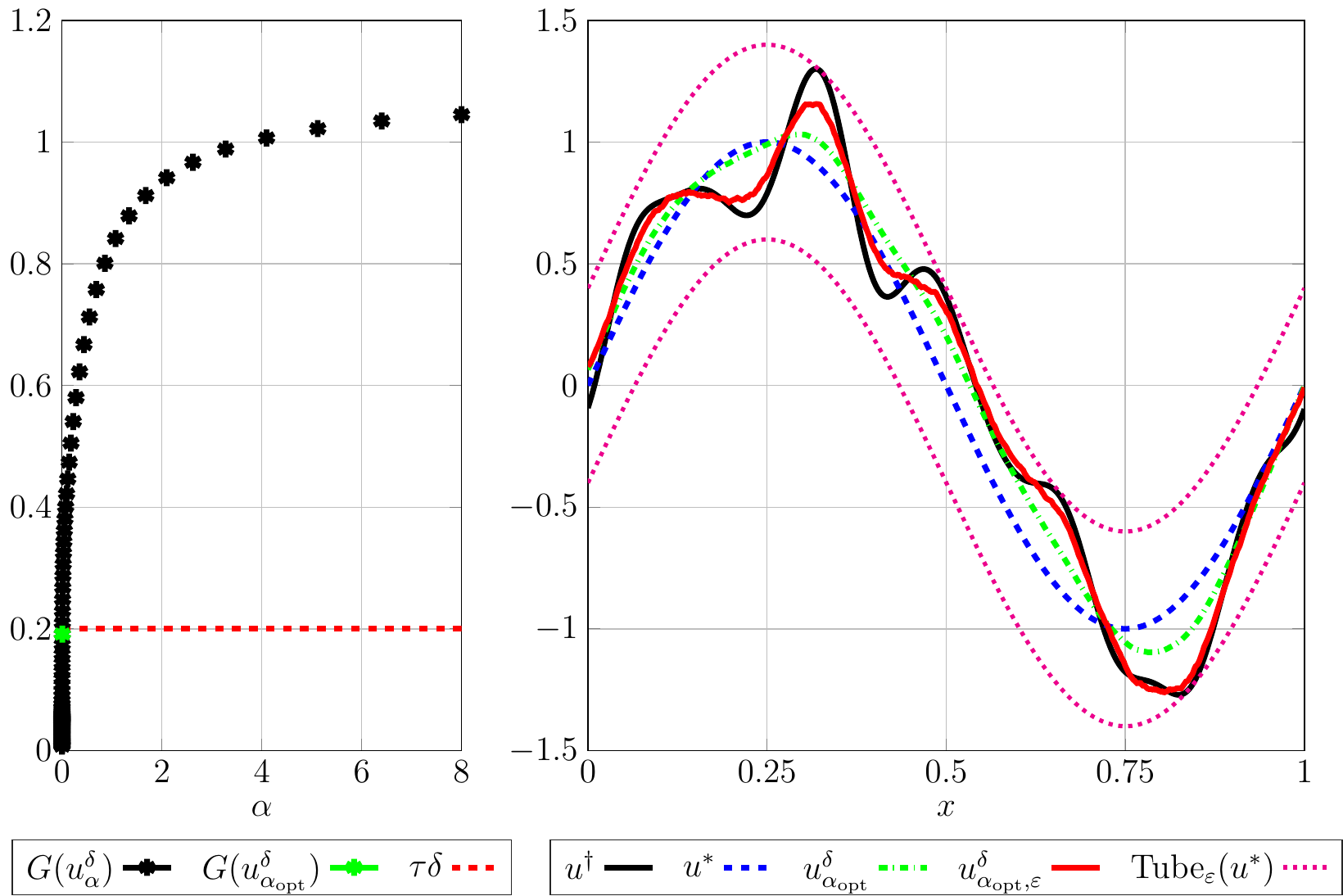}
   \caption{Morozov's discrepancy principle for $\delta=0.1$, $\tau=2$, $\ve=0.4$ and $q=2$. On the left we have the discrepancy values for different values of $\alpha$ and the one for $\alpha_{\mathrm{opt}}=0.0063$ (in green). On the right we plot the reconstructed $u_{\alpha_{\mathrm{opt}},\ve}^\delta$, $u_{\alpha_{\mathrm{opt}}}^\delta$ for $\alpha_{\mathrm{opt}}$ in comparison to $\ud$ and $u^*$.}
    \label{F:figure9}
\end{figure}

\begin{figure}
    \centering
    \includegraphics[width=0.7\textwidth]{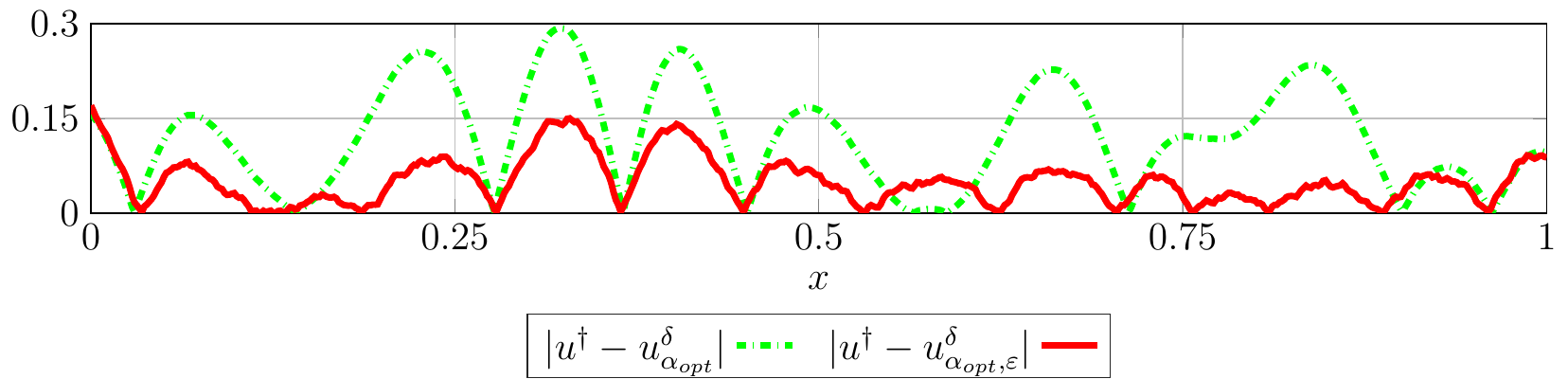}
    \caption{Absolute error of the reconstructions $u_{\alpha_{\mathrm{opt},\ve}}^\delta$ and $u_{\alpha_{\mathrm{opt}}}^\delta$ shown in~\Fref{F:figure9}.}
    \label{F:figure10}
\end{figure}

As can be seen in \Fref{F:figure7} and \Fref{F:figure9}, we obtain improved approximations of the true solution. In each of these figures our solution (in red) fits better the true solution (in black) than the one computed by the minimization of $\Jvd$ (in green). Moreover, in~\Fref{F:figure8} and~\Fref{F:figure10}, respectively, we compare the absolute error of these reconstructions with respect to the true solution $\ud$. Both example cases show that tolerances can indeed advance the quality of the approximation but one has to further examine under which scenarios this happens.

A drawback of the discrepancy principle is that it tends to select small values of $\alpha$ as the optimal one, which doesn't promote the use of tolerances. However, when the noise in the data is larger, the use of the discrepancy principle as the parameter choice rule makes the regularization stronger. Especially in the case of $\ve \geq \delta \geq \alpha$ we obtain better results. This can be seen in both figures. In contrast to the classical regularization, here the tolerances do not allow the reconstruction to rely solely on the reference solution.

The previous results were produced using the optimal regularization parameter for the generalized Tikhonov functional $\Jvd$. However, we can also use the discrepancy principle directly for finding the optimal regularization parameter for our functional $\Jevd$, i.e., implementing the discrepancy principle 
\[G(u_{{\alpha,\ve}}^\delta) := \norm{K u_{\alpha,\ve}^\delta - \vd} \leq \tau \delta, ~ \text{ with } ~ \uade := \arg\min \Jevd(u).\] 
This is shown in \Fref{F:figure11} where the discrepancy principle was used for finding the optimal regularization parameter $\alpha_{\mathrm{opt}} = 0.6872$ for $\tau = 4$, $\ve = 0.5$ and $q=2$. In the right plot of the figure we observe that indeed tolerances enhance the quality of our reconstruction. In~\Fref{F:figure12_error} we compare the absolute error of the generalized Tikhonov minimizer and our solution with respect to the ground truth and confirm that our approximation is closer to $\ud$. 
\begin{figure}
    \centering
    \includegraphics[width=0.7\textwidth]{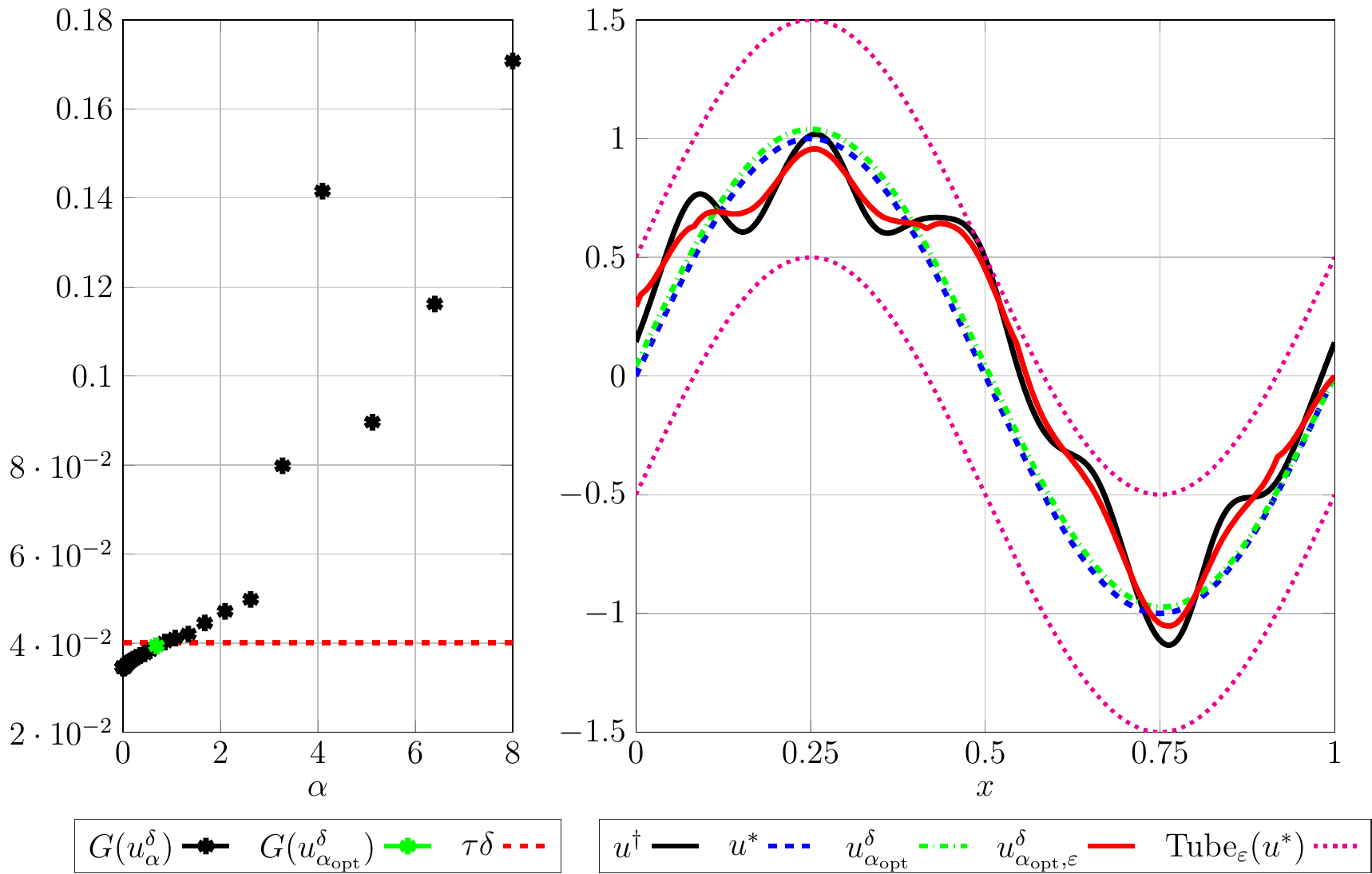}
    \caption{Morozov's discrepancy principle for $\delta=0.01$, $\tau=4$, $\ve=0.5$ and $q=2$. On the left we have the discrepancy values for different values of $\alpha$ and the one for $\alpha_{\mathrm{opt}}=0.6872$ in green. On the right we plot the reconstructed $u_{\alpha_{\mathrm{opt}},\ve}^\delta$ and $u_{\alpha_{\mathrm{opt}}}^\delta$ for $\alpha_{\mathrm{opt}}$ in comparison to $\ud$ and $u^*$.}
    \label{F:figure11}
\end{figure}

\begin{figure}
    \centering
    \includegraphics[width=0.7\textwidth]{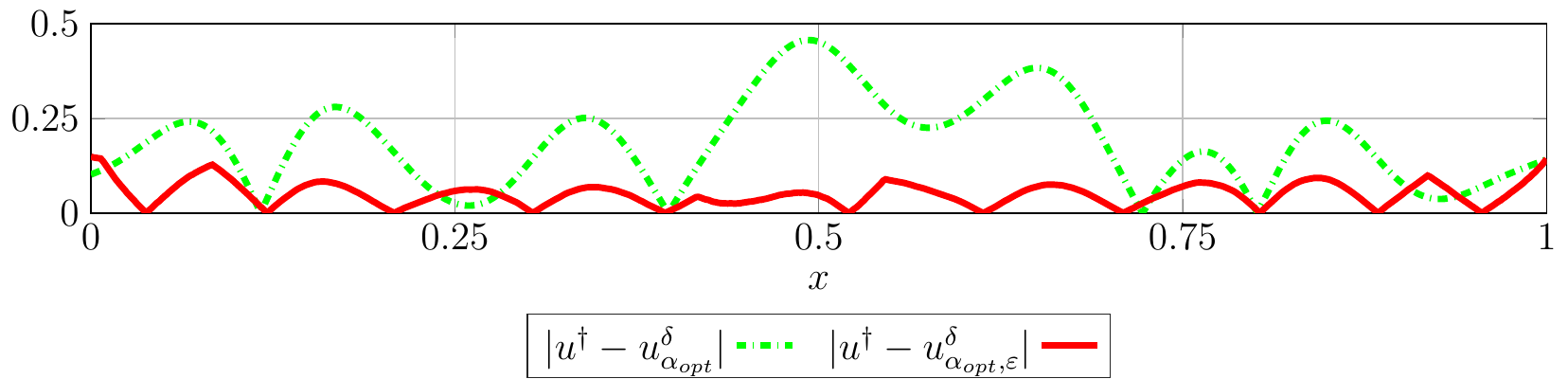}
    \caption{Absolute error of the reconstructions $u_{\alpha_{\mathrm{opt},\ve}}^\delta$ and $u_{\alpha_{\mathrm{opt}}}^\delta$ shown in~\Fref{F:figure11}.}
    \label{F:figure12_error}
\end{figure}

In the left part of~\Fref{F:figure11}, we also observe that the values of the discrepancy principle $G(\uade)$ are not monotonically increasing as expected for linear operator problems. This means that the existence of an optimal $\alpha$ satisfying the discrepancy principle might not be guaranteed~\cite{Anzengruber_Ramlau2009}. This phenomenon also indicates that $\tau$ should be larger than (here we chose $\tau=4$) so that the discrepancy principle is surely satisfied. These are only some first results on how to choose the regularization parameter when incorporating tolerances. Of course, other strategies can be considered such as a relaxation of the discrepancy principle as proposed in~\cite{Anzengruber_Ramlau2009,Ramlau2002}, a generalization of the $L$-curve~\cite{Belge_2002} or other heuristic rules that have been proposed in~\cite{JinLorenz_2010}. Therefore, this topic is still open for further investigation.

\section{Remarks on tolerances and sparsity}

A question that arises naturally is whether sparse solutions can be promoted when tolerances are incorporated in the regularization term. Sparsity constrained Tikhonov regularization is well-studied, for example in~\cite{DDD} and~\cite{Grasmair2008}, and often it is imposed with the use of a (weighted) $\ell^q$-norm of the coefficients of $u$ with respect to a given orthonormal basis for $U$. Here, we examine the possibility of obtaining sparse solutions in our setting with tolerances, too. 

The true solution $\ud$ of \eref{invpb} is sparse if there exist only a finite number of non-zero coefficients with respect to the chosen orthonormal basis. Following the classical approach for sparsity, we consider an orthonormal basis ${\{\phi_i\}}_{i \in \N} \subset U$ and with coefficients $\cie := \inner{d_\ve(u),\phi_i}$ we define the regularization functional
\begin{equation*} 
    \Rqe(u) = \sum_{i \in \N} \abs{\inner{d_\ve(u),\phi_i}}^q \quad \text{for } \quad 1 \leq q \leq 2.
\end{equation*}
Assuming that the minimization of the generalized Tikhonov functional 
\begin{equation}\label{E:Tkhn_Sparse}
    \Jvd(u) = \norm{F(u) - \vd}^p + \alpha \sum_{i \in \N} \abs{\inner{u,\phi_i}}^q
\end{equation}
yields a sparse regularized solution $\uad$, we investigate if the same is true for the functional
\begin{equation}
    \Jevd(u) = \norm{F(u) - \vd}^p + \alpha \sum_{i \in \N} \abs{\inner{d_\ve(u),\phi_i}}^q.
\end{equation}
This, basically, means that we examine the connection between the coefficients $\ci := \inner{u,\phi_i}$ and $\cie = \inner{d_\ve(u),\phi_i}$.
Let us have a look at a coefficient $\cie$, it is given as
\begin{eqnarray}
    \cie = \inner{\de(u),\phi_i} &= 
    \cases{
        \inner{\abs{u} -\ve, \phi_i} &$\abs{u}>\ve$\\ 
        0 &$\abs{u}\leq \ve$ 
    } \nonumber \\ &= 
    \cases{
        -(\inner{u,\phi_i} + \inner{\ve, \phi_i}) &$u < -\ve$\\ 
        0 &$\abs{u}\leq \ve$ \\ 
        \inner{u,\phi_i} - \inner{\ve,\phi_i} &$u > \ve$
    } \nonumber \\ &=
    \cases{
        -(\ci + \inner{\ve, \phi_i}) &$u < -\ve$\\ 
        0 &$\abs{u} \leq \ve$ \\ 
        \ci- \inner{\ve,\phi_i} &$u > \ve$
    }.\label{Eq:coeffsExpression}
\end{eqnarray}
Since $u$ and $\de(u)$ can be two functions that differ significantly based on the value of $\ve$, they will naturally have different coefficients as well. Therefore, we do not aim to compare the coefficient values but to examine if the application of tolerances can additionally enhance sparsity. The sparsity assumption on the coefficients $\ci$ is made to confirm if $\cie$ can be sparse or even, sparser. From the expression in~\eref{Eq:coeffsExpression} it is easy to check that if a coefficient $\ci$ is zero, it is rather improbable that $\cie$ will also be zero. For a function that is known to be sparse on the chosen basis, this means that with the tolerance assumption the solution does not remain sparse.

With the above discussion, we conclude that when applying the tolerance function, sparsity is lost as some of the initially inactive coefficients most likely are shifted away from zero. To illustrate this effect, we present a simple example.

\subsection{Example} 
We assume the function $u(t) = \sin(t)$ with $t \in [-2\pi, 2\pi]$ and we consider the Fourier basis, which is commonly used for the approximation of $2\pi$-periodic functions. We apply the $\ve$-insensitive distance on $u$ and then compute the Fourier approximations of $u$ and $d_\ve(u)$ using the Fourier series expansion. Our aim is to examine the sparsity of the computed Fourier coefficients. In our example we use a constant tolerance area around zero denoted by $\mathrm{Tube}_\ve(0) = \{ t \in [-2\pi,2\pi] ~ \text{s.t.}~ -\ve \leq y(t) \leq \ve\}$. Moreover, we know that the Fourier approximation of the sine function has only one non-zero coefficient, which is equal to $1$. 

In \Fref{F13a}, we plot the graphs of $u$ and $d_\ve(u)$ as well as the tolerance area which is shaded in pink. We symbolize the Fourier approximations $\hat{u},\hat{d_\ve}(u)$ and denote their Fourier coefficients by $\alpha_n,\beta_n$ and $\alpha_n^\ve, \beta_n^\ve$, respectively. \Fref{F13b} shows the computed Fourier coefficients using the first $20$ terms in the Fourier series. The figure illustrates what was previously mentioned, namely, that the approximation of the sine function has only one non-zero Fourier coefficient while for approximating $d_\ve(u)$ there are more non-zero coefficients, which indicates that $d_\ve$ is not sparse in $u$.
\begin{figure}
    \centering
    \subfloat[]{\label{F13a}
    \includegraphics[width=0.45\textwidth]{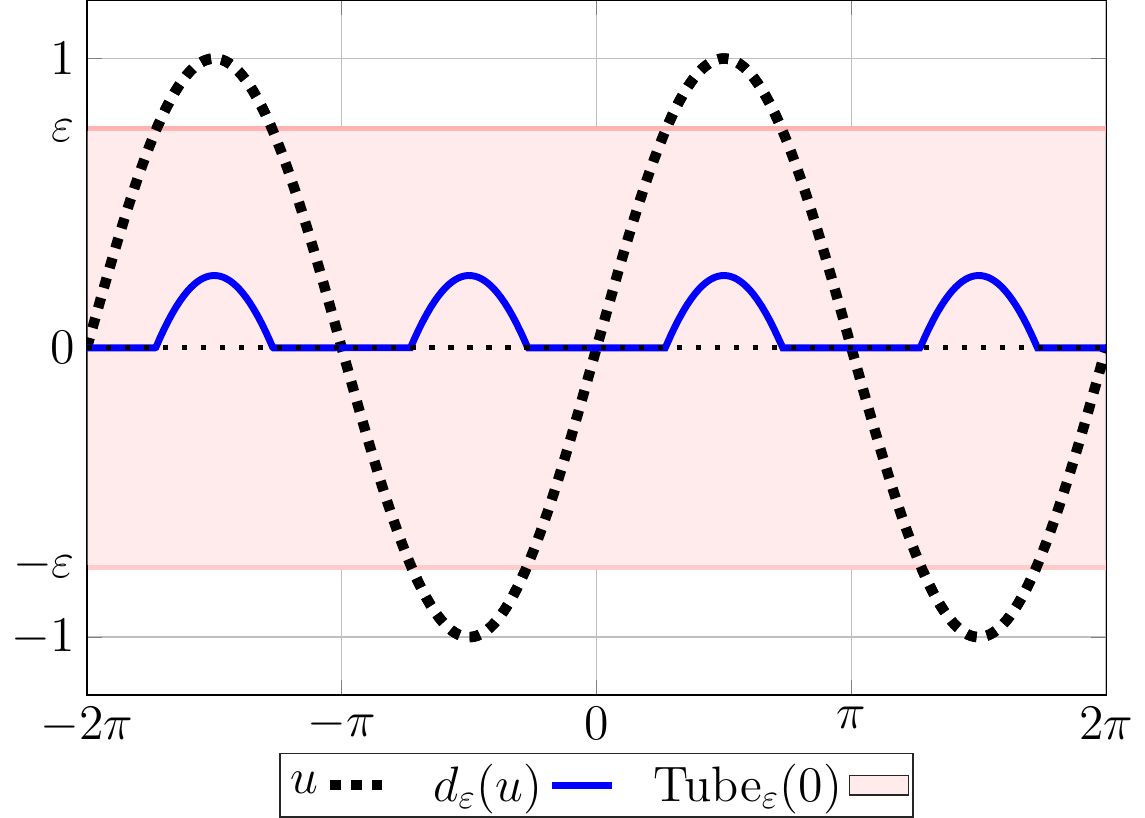}}%
	\subfloat[]{\label{F13b}
	\includegraphics[width=0.45\textwidth]{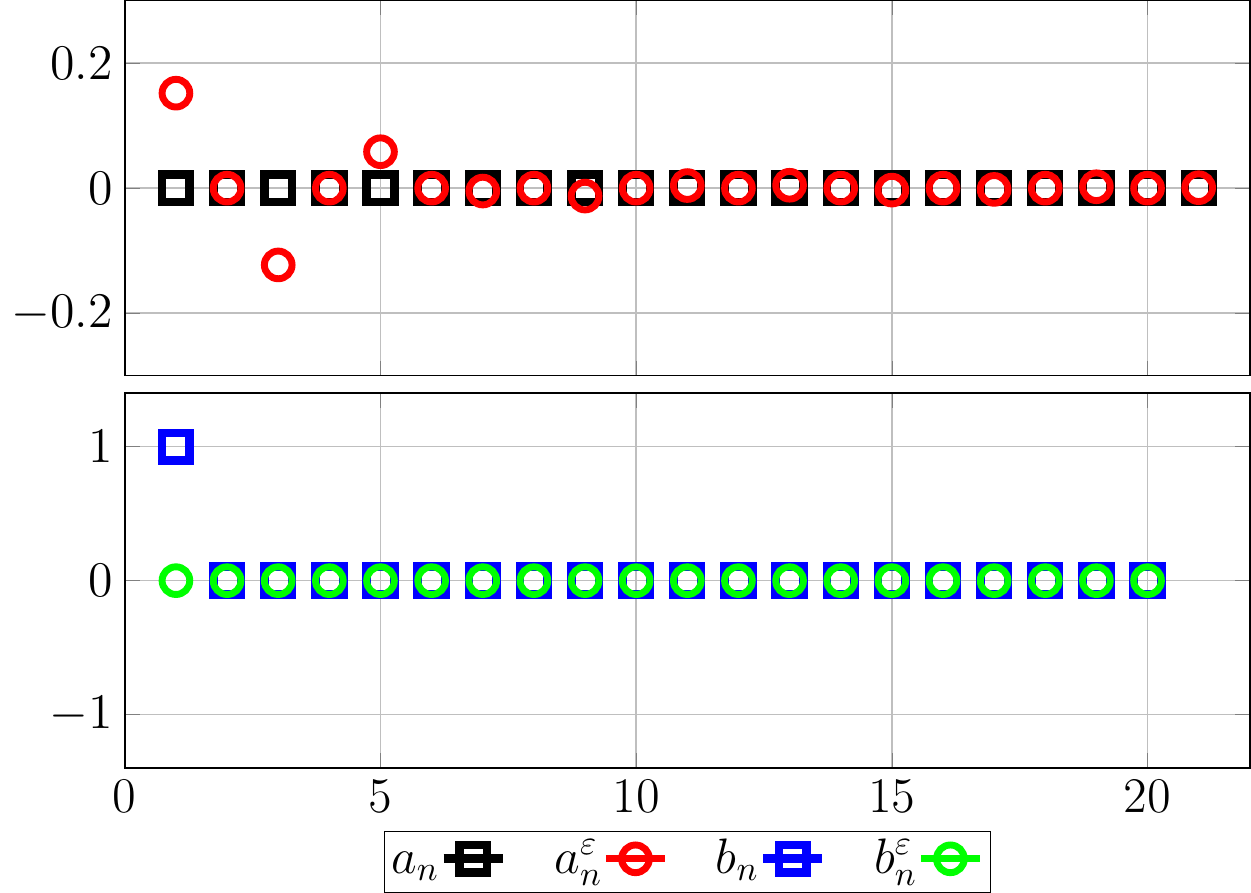}}
	\caption{The functions $u$ and $d_\ve(u)$ plotted together with a constant tolerance area defined for $\ve = 0.75$ are shown in (a). The first $20$ coefficients from the Fourier approximations of $u$ and $d_\ve(u)$ are shown in (b).}
	\label{F:Figure13}
\end{figure}

The fact that the tolerance function is does not promote sparse reconstructions does not necessarily mean that our approach cannot prove useful. Moreover, when sparsity is required one can adopt an alternative approach by following the idea of the elastic net regularization, see for instance~\cite{JinElasticNet2009,Schiffler2010,Zou_Hastie_2005}. By doing so, both tolerances and sparsity constraints can be taken into account in the solution.

\section{Conclusion}

In this paper we discuss a modified Tikhonov functional with tolerances in the regularization term that allow for small deviations to be included in the solution. The existence, stability and weak convergence of minimizers for this functional is proved, as well as the well-posedness of minimizers when the tolerances converge to zero, that is, when returning to the generalized Tikhonov approach. In addition, we provide convergence rates of the minimizers in the Bregman distance. 

The theoretical analysis is followed by numerical results on an academic example. These results confirm that improved reconstructions are possible when tolerances are used in the regularization term. We also discuss parameter choice rules that can be used in combination with the appropriate tolerances for better fitting the true solution. As our approach is rather new, the parameter choice rules need to be further investigated in order to clarify what best fits our framework. 

Depending on the problem in hand, the structure of the solution may not need to be sparse. In this case, our approach is valid as its well-posedness is proved. On the other hand, if sparsity is required, we can still achieve it by introducing an $\ell^1$-norm penalty term in our functional. This approach is motivated by the elastic net regularization that has previously been used in~\cite{JinElasticNet2009,Schiffler2010,Zou_Hastie_2005}. In this work the authors use a second penalty term (in the $\ell^2$-norm) to their sparsity-promoting Tikhonov functional in order to guarantee stability. The idea of adapting our functional to the elastic-net approach is a topic of further consideration and future work.

\ack{The authors gratefully acknowledge financial support by the German Research Foundation (DFG) within the framework of RTG 2224/1 ``$\pi^3$: Parameter Identification --- Analysis, Algorithms, Applications''. 
The authors also wish to thank Phil Gralla for providing the subgradient algorithm that was used for the minimization of the functional in the numerical examples.
}

\section*{References}


\end{document}